\documentclass[review,onefignum,onetabnum]{siamart171218}

\usepackage{lipsum}
\usepackage{amsfonts}
\usepackage{graphicx}
\usepackage{epstopdf}
\usepackage{algorithmic}
\ifpdf
  \DeclareGraphicsExtensions{.eps,.pdf,.png,.jpg}
\else
  \DeclareGraphicsExtensions{.eps}
\fi

\newsiamremark{remark}{Remark}
\newsiamremark{hypothesis}{Hypothesis}
\crefname{hypothesis}{Hypothesis}{Hypotheses}
\newsiamthm{claim}{Claim}

\headers{Randomized linear solvers}{V. Kalantzis, Y. Xi, L. Horesh and Y. Saad}
\headers{Straggler-tolerant linear solvers}{V. Kalantzis, Y. Xi, L. Horesh and Y. Saad}

\title{Randomized linear solvers for computational architectures with straggling workers\thanks{Submitted to the editors on June 30, 2024.}}
\title{Straggler-tolerant stationary iterative methods for
linear systems\thanks{Submitted to the editors on October 8, 2024.}}

\author{
Vassilis Kalantzis\thanks{IBM Research, 
Thomas J. Watson Research Center, Yorktown Heights, NY USA 
  (\email{vkal@ibm.com, lhoresh@us.ibm.com}).}
\and Yuanzhe Xi\thanks{Department of Mathematics, Emory University, Atlanta, GA, 
  (\email{yxi26@emory.edu}).} Research supported by NSF  DMS 2208412 and RTG DMS 2038118.
\and Lior Horesh\footnotemark[2]
\and Yousef Saad\thanks{Department of Computer Science and Engineering, University of Minnesota, Minneapolis, MN, 
  (\email{saad@umn.edu}). Research supported by the NSF award DMS 2208456.}
}

\usepackage{amsopn}

\makeatletter
\newcommand*{\addFileDependency}[1]{
  \typeout{(#1)}
  \@addtofilelist{#1}
  \IfFileExists{#1}{}{\typeout{No file #1.}}
}
\makeatother

\newcommand*{\myexternaldocument}[1]{%
    \externaldocument{#1}%
    \addFileDependency{#1.tex}%
    \addFileDependency{#1.aux}%
}

\ifpdf
\hypersetup{
  pdftitle={Flexible GMRES},
  pdfauthor={V. Kalantzis, L. Horesh, Y. Saad and Y. Xi}
}
\fi

\newcommand{\norm}[1]{\left\lVert#1\right\rVert}
\usepackage{amsmath}

\usepackage{amssymb}

\usepackage{multirow}
\usepackage{booktabs} 
\newcommand{\mybinom}[2]{\Bigl(\begin{array}{@{}c@{}}#1\\#2\end{array}\Bigr)}
\usepackage{pgfplots}
\usepackage{tikz}

\myexternaldocument{ex_supplement}

\begin{document}
\nolinenumbers
\maketitle

\begin{abstract}
In this paper, we consider the iterative solution of linear algebraic equations under the condition that matrix-vector products with the coefficient matrix are computed only partially. At the same time, non-computed entries are set to zeros. We assume that both the number of computed entries and their associated row index set are random variables, with the row index set sampled uniformly given the number of computed entries. This model of computations is realized in hybrid cloud computing architectures following the controller-worker distributed model under the influence of straggling workers. We propose {straggler-tolerant} Richardson iteration scheme and Chebyshev semi-iterative schemes, and prove sufficient conditions for their convergence in expectation. Numerical experiments verify the presented theoretical results as well as the effectiveness of the proposed schemes on a few sparse matrix problems. 
\end{abstract}

\begin{keywords}
  Richardson iteration, Chebyshev iteration, controller-worker architectures, straggling, randomization, cloud computing.
\end{keywords}

\begin{AMS}
  65F08, 65F10, 68M15, 68Q87.
\end{AMS}

\section{Introduction}

The use of on-demand remote computer resources (cloud computing) is becoming increasingly 
a mainstream alternative for solving large-scale scientific problems in businesses and academia due 
to its scalability and cost efficiency \cite{huber2020aiida,rahhali2024parallel,talirz2020materials}. A particular instance of cloud computing, termed hybrid cloud, offers additional flexibility by combining on-premises computing infrastructure with a public cloud formed by remote (non-dedicated) processing elements which are allocated dynamically subject to considerations such as cost and latency \cite{coyne2018ibm}. 
A hybrid cloud generally follows a controller-worker model of asymmetric communication 
where the controller typically resides on the on-premises infrastructure and is responsible for task distribution, synchronization, monitoring, and management of workers, while the workers receive data, perform computations, and send data back to the controller. 

One limitation of controller-worker models implemented on cloud computing infrastructures is 
the phenomenon of straggling \cite{hussain2019sla}. Straggling workers refer to those processes 
that complete their workload significantly slower than their peers and thus delay the overall 
flow of computations \cite{wang2015using}. Specifically, in the context of  iterative solvers, straggling frequently arises during the computation of matrix-vector products that involve the iteration matrix. A simple remedy to this problem is the allocation of a fixed amount of time in which
each worker needs to return its local product otherwise a zero is placed instead \cite{amiri2019computation,pmlr-v238-kalantzis24a}. While this approach reduces idle wait, it 
introduces complexities when using classical iterative subspace solvers because most convergence
analyses for these solvers assume that matrix-vector products are computed exactly up to the 
round-off error. In particular, while the behavior of Krylov iterative solvers with inexact matrix-vector products and/or faults has been studied, e.g., see \cite{bouras2005inexact,bridges2012fault,coleman2017comparison,coleman2021fault,elliott2014evaluating,giraud2007convergence,gratton2019exploiting,hoemmen2011fault,langou2008recovery,sao2013self,shantharam2012fault,simoncini2003theory,simoncini2005occurrence,van2004inexact}, the inability to compute exact matrix-vector products often
results in delayed convergence \cite{agullo2020exploring,sidje2011evaluation,sleijpen2005restarted}.

The above discussion motivates the study of the straggler-tolerant iterative solution of a system of 
linear algebraic equations $Az=v$ on controller-worker architectures subject to the 
constraint that matrix-vector products with the $N\times N$  matrix $A$ are 
almost always computed partially, i.e., only a subset of the entries is returned and the 
omitted entries are set equal to zero. More specifically, we assume that the matrix-vector 
product $Af$ between $A$ and a vector $f\in \mathbb{R}^N$ is replaced with an oracle that 
returns a random set of $T\in \mathbb{N}$ entries of $Af$ indexed by ${\cal T}\subseteq \{1,2,\ldots,N\},\ |{\cal T}|=T$. 
{It is also possible to consider a column-wise distribution of the matrix-vector product, however in this paper we focus on the row-wise model due to its simplicity as well as the fact that the controller only needs to receive at most one scalar as opposed to an $N$-length vector per non-straggler worker.} 

Throughout the rest of this paper, we assume both the number of observed entries $T$ and the corresponding subset ${\cal T}$ of 
observed indices are random variables. We further assume 
that the probability of observing each outcome of ${\cal T}$ is uniform for a given $T$.  
{Though assuming conditional independence of ${\cal T}$ is essential 
to the development of the theoretical framework, we note that 
this assumption can be restricting in practical applications and further studies are 
required to cover the existing knowledge gap. For example, when a worker is consistently 
slower, the associated index might never be included in ${\cal T}$.} 
{Moreover, the analysis presented in this paper  only applies when one worker is responsible for a single entry of each matrix-vector 
product. The more practical scenario where a worker handles multiple contiguous index rows, i.e., the number of workers is less than the dimension $N$, is generally more complex 
and requires a separate study that is left as future work.}

In this paper, we consider the Richardson iteration and Chebyshev semi-iterative schemes, and 
focus on their behavior when classical matrix-vector products are replaced with partial matrix-vector products as outlined above. Our main contributions are summarized as follows: $a$) 
We demonstrate that the expected value of the approximate solution at each iteration of the {straggler-tolerant} Richardson iteration is equal to the iterate produced by the classical Richardson iteration, provided that two specific scalar parameters are used in the {straggler-tolerant} version. Furthermore, we demonstrate that the variance of the iterate of {straggler-tolerant} Richardson iteration generally increases as the iteration number increases. 
$b$) We extend the framework of {straggler-tolerant} Richardson to the stationary Chebyshev semi-iterative method, a form of second-order iteration, and show that the iterates of the {straggler-tolerant} variant are -in expectation- equal to the corresponding iterate produced by the classical variant. 
$c$) {Our numerical experiments illustrate that both the straggler-tolerant Richardson and
Chebyshev semi-iterative methods can converge in expectation to the true solution of
the linear system, and that the hindrance to convergence due to missing contributions
from straggling workers can be indeed mitigated.}

The structure of this paper is as follows. Section \ref{sec2} discusses in greater detail 
the problem of stragglers and introduces our model of computations and its motivation. 
Section \ref{sec4} presents a probabilistic analysis of the convergence of 
{straggler-tolerant} Richardson iteration with partially complete 
matrix-vector products. Section \ref{sec:cheby} proposes 
the {straggler-tolerant} Chebyshev semi-iterative method. Section 
\ref{sec5} presents numerical illustrations and comparisons. Finally, 
Section \ref{sec6} gives our concluding remarks. We denote by $\mathbb{E}$ 
the expectation of a random variable. Also, we denote by $e_i$ the $i$th 
column of the $N\times N$ identity matrix, and $1^N$ the vector of length 
$N$ with all ones. Finally, the $i$th entry of the vector $x$ will be 
denoted by $[x]_i$.

\section{A model for partially complete matrix-vector products} \label{sec2}

The work presented in this paper is mainly motivated by the phenomenon of straggling in 
hybrid cloud computing environments operating under the controller-worker computational 
model. In this section, we define a model for the matrix-vector product realized in the 
presence of straggling workers.

\subsection{The problem of straggling workers}

In the controller-worker model, the controller is responsible for gathering and 
processing the elements produced by the worker entities. Each worker entity (process) 
typically exploits a separate processing element of hardware and executes in parallel and 
independently from the rest of the workers. When all workers require the same amount of 
time to execute their tasks, a controller-worker model can enhance granularity and reduce 
the wall-clock time of an application. In practice each worker generally requires an amount 
of time that varies considerably from other workers, leading to a phenomenon known as 
straggling. Straggling in distributed computing refers to the phenomenon where some workers are unresponsive or take significantly longer to complete their tasks compared to others, thus 
leading to delays in the overall completion time of distributed computations. Such workers 
are known as stragglers \cite{dean2013tail} and they degrade the parallel efficiency of 
distributed systems. 

In numerical linear algebra,  matrix-vector products are commonly performed in parallel to accelerate the execution of iterative solvers for large linear systems \cite{XU2022102956}. Assume under the controller-worker model, the $i$th entry of the $N\times 1$ matrix-vector product $y=Af$ between a $N\times N$ matrix $A$ and a $N\times 1$ vector $f$ is computed by assigning the $i$th worker the computation of the scalar product between the vector $f$ and the $i$th row of $A$. Each worker performs its respective task independently while the controller aggregates the individual scalars produced by each of the $N$ workers. Nonetheless, it is generally impossible to determine a priori how long each worker might execute until it returns its part of the matrix-vector product $y=Af$; especially when the workers are not dedicated to a particular application and are 
distributed across several geographical regions as is likely in cloud computing infrastructures. Straggling becomes increasingly more likely for larger values of $N$, since, even when the probability that each worker slows down or becomes unresponsive is small, the chance that at least one worker becomes a straggler increases, and so does the expected latency of the iterative solver. 

\subsection{Matrix-vector products with omitted entries}

A worker that becomes a straggler in the current iteration of an iterative solver 
is not necessarily a straggler in a future iteration and vice versa. For example, in 
matrix-vector products with matrices whose rows have roughly equal numbers of non-zero entries, straggling is typically attributed to short-time network contention and latency. Therefore, an iterative solver that aims to mitigate straggling should assume
that neither the number $T$ of straggling workers at a given iteration nor their corresponding index set ${\cal T}$ remains fixed.
In this paper, we aim to develop a flexible framework where both quantities are random variables. 

Let $A$ be a $N\times N$ matrix and consider a random integer $T$ bounded by $1$ 
from below and $N$ from above. Furthermore, let ${\cal T} \subseteq \{1,2,\ldots,N\}$ 
denote a random subset of rows of $A$ of cardinality $T$. In the following, we define the 
matrix-vector product operator ‘$\times_{{\cal T}}$'.
\begin{definition}\label{def:one}
Let $T\in \mathbb{N}$ denote a random integer taking values in the closed interval 
$[1,N]$ and ${\cal T}$ denote a random subset of $T\in \mathbb{N}$ integers without replacement from the integer set $\{1,2,\ldots,N\}$. We define the  matrix-vector product 
$y = A\times_{{\cal T}}f$ between the matrix $A$ and a vector $f\in \mathbb{R}^N$ such 
that the $i$th entry of $y\in \mathbb{R}^N$ is equal to: 
\begin{equation}\label{eq:mvm0}
   [y]_i = 
    \begin{cases}
    \sum_{j=1}^{j=N} A_{ij}f_j  &\ \mathrm{if}\ i\in {\cal T}\\
    0 &\ \mathrm{if}\ i\notin {\cal T}
    \end{cases}.
\end{equation} 
\end{definition}
Unless mentioned otherwise, throughout the rest of this paper we assume that the random variable $T$ takes any of the values $1,2,\ldots,N$ following a certain distribution, and the random subset of ${\cal T}$ picks any 
$T\equiv |{\cal T}|$ integers of $\{1,2,\ldots,N\}$ with equal probability, i.e., 
each one of the ${\mybinom{N}{T}}$ possible row sets of $A$ is picked with probability ${\mybinom{N}{T}}^{-1}$ \cite{teke2018asynchronous,teke2019random}. 

Consider now the diagonal random matrix formed by the summation of $T$ canonical 
outer products
\begin{equation*}
D_{\cal T} = \sum\limits_{i \in {\cal T}}e_ie_i^{\top}.
\end{equation*} 
Equation (\ref{eq:mvm0}) can be written in the equivalent form 
\begin{align*}
        y = D_{\cal T}Af.
\end{align*}
Notice that when $T\equiv N$, as in the classical case, the matrix $D_{\cal T}$ is equal to 
the $N\times N$ identity matrix, and $y = A\times_{{\cal T}}f = Af$. Figures \ref{cartoon1}
and \ref{cartoon2} visualize two matrix-vector multiplications $y=D_{\cal T}Af$ using the 
controller-worker model where $N=4$ and ${\cal T}=\{1,4\}$ or ${\cal T}=\{3\}$, respectively.

\begin{remark}
{The matrix-vector product model presented in this section does not require $A$ to be neither explicitly formed nor sparse. The only assumption we make is the row-wise distribution of workers in computing the matrix-vector products.}
\end{remark}

\subsection{Relation to asynchronous models}

The equation defined in (\ref{eq:mvm0}) computes the exact entry of the matrix-vector product 
depending on whether the corresponding index belongs to the row index subset $\cal T$. 
This concept is akin to the principles of asynchronous iterative algorithms used in computing the stationary point $z=G(z),\ G:\mathbb{R}^N\rightarrow \mathbb{R}^N$, where the 
$i$th entry of the vector $z$ satisfies $[z]_i=G_i(z),\ i=1,\ldots,N$. Asynchronous approaches are particularly advantageous in distributed-memory systems as they minimize idle time across processing elements by reducing synchronization. An asynchronous method for computing the stationary point $z$ can be defined mathematically as
\begin{equation}\label{eq:mvm1}
[z]_i^{k}=\begin{cases}
            G_i\left([z]_1^{s_1(k)},\ldots,[z]_N^{s_N(k)}\right) &\ \mathrm{if}\ i\in {\cal T}_k\\
            [z]_i^{k-1} &\ \mathrm{if}\ i\notin {\cal T}_k\\
\end{cases},
\end{equation}
where $[z]_i^{k}$ denotes the $i$th component of the iterate at time instant $k$, ${\cal T}_k$ is the set of indices updated at instant $k$, and $s_j(k)$ is the last instant the $j$th component was updated before being read at instant $k$ \cite{bertsekas2015parallel,frommer2000asynchronous,wolfson2019modeling}. The increasing gap between the time required to share a floating-point number between different processing elements and the time needed to perform a single floating-point operation by one of the processing elements has led to a revived interest in the analysis and application of asynchronous algorithms in numerical linear algebra  \cite{avron2015revisiting,chazan1969chaotic,frommer2000asynchronous,hook2018performance,glusa2020scalable,teke2018asynchronous,teke2019random,wolfson2019asynchronous}. Moreover, while synchronous stationary solvers require the spectral radius of the iteration matrix to be less than one, asynchronous variants can achieve convergence even when the spectral radius exceeds one. This is because they typically operate on a submatrix of the iteration matrix which might have more favorable properties \cite{wolfson2018convergence,wolfson2019modeling}. 

The algorithms discussed in this paper are fully synchronous and our main objective is to 
successfully solve  linear systems under the constraints in (\ref{eq:mvm0}) rather 
than reduce latency. One notable difference between the models defined by (\ref{eq:mvm0}) 
and (\ref{eq:mvm1}) is that the former does not exploit stale information but instead sets 
any entry not indexed in ${\cal T}_k$ equal to zero.\footnote{A model similar to 
the one defined by (\ref{eq:mvm0}), termed as an “asynchronous method without communication delays", 
has been considered as a special case of asynchronous computing in \cite{wolfson2019modeling}.} 
While a fully asynchronous approach can lead to enhanced computational-communication overlap 
and reduce latency, e.g., see for example \cite{chow2021asynchronous} for asynchronous 
Richardson, our choice to follow (\ref{eq:mvm0}) leads to a simple update formula that we 
analyze in the next two sections.

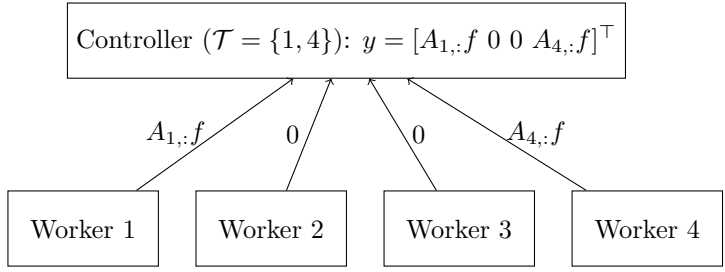
\begin{figure}
\centering
\begin{tikzpicture}
    \node[draw, rectangle, minimum width=2cm, minimum height=1cm] (controller) at (-1.5,0) {Controller (${\cal T}=\{1,4\}$): $y = [A_{1,:}f\ 0\ 0\ A_{4,:}f]^{\top}$};

    \node[draw, rectangle, minimum width=2cm, minimum height=1cm] (worker1) at (-5,-2.5) {Worker 1};
        \node[draw, rectangle, minimum width=2cm, minimum height=1cm] (worker2) at (-2.5,-2.5) {Worker 2};
    \node[draw, rectangle, minimum width=2cm, minimum height=1cm] (worker3) at (0,-2.5) {Worker 3};
        \node[draw, rectangle, minimum width=2cm, minimum height=1cm] (worker4) at (2.5,-2.5) {Worker 4};

    \draw[<-] (controller) -- (worker1) node[midway,left] {$A_{1,:}f$};
    \draw[<-] (controller) -- (worker2) node[midway,left] {$0$};
    \draw[<-] (controller) -- (worker3) node[midway,right] {$0$};
    \draw[<-] (controller) -- (worker4) node[midway,right] {$A_{4,:}f$};
\end{tikzpicture}
\caption{Matrix-vector multiplication $y=D_{\cal T}Af$ under the controller-worker model for a toy example with $N=4$ and $T=2,\ {\cal T}=\{1,4\}$.}\label{cartoon1}
\end{figure}

\begin{figure}
\centering
\begin{tikzpicture}
    \node[draw, rectangle, minimum width=2cm, minimum height=1cm] (controller) at (-1.5,0) {Controller (${\cal T}=\{3\}$): $y = [0\ 0\ A_{3,:}f\ 0]^{\top}$};

    \node[draw, rectangle, minimum width=2cm, minimum height=1cm] (worker1) at (-5,-2.5) {Worker 1};
    \node[draw, rectangle, minimum width=2cm, minimum height=1cm] (worker2) at (-2.5,-2.5) {Worker 2};
    \node[draw, rectangle, minimum width=2cm, minimum height=1cm] (worker3) at (0,-2.5) 
    {Worker 3};
    \node[draw, rectangle, minimum width=2cm, minimum height=1cm] (worker4) at (2.5,-2.5) {Worker 4};

    \draw[<-] (controller) -- (worker1) node[midway,left] {$0$};
    \draw[<-] (controller) -- (worker2) node[midway,left] {$0$};
    \draw[<-] (controller) -- (worker3) node[midway,right] {$A_{3,:}f$};
    \draw[<-] (controller) -- (worker4) node[midway,right] {$0$};
\end{tikzpicture}
\caption{Matrix-vector multiplication $y=D_{\cal T}Af$ under the controller-worker model for a toy example with $N=4$ and $T=1,\ {\cal T}=\{3\}$.}\label{cartoon2}
\end{figure}
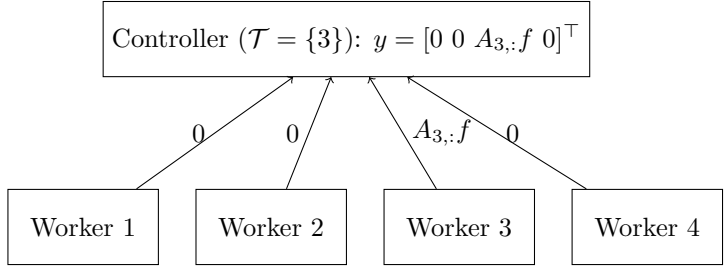

\section{Richardson iteration with straggling workers} \label{sec4}

In this section, we analyze Richardson stationary iteration with row 
sampling to solve the linear system $Az=v,\ v\in \mathbb{R}^N$. 

Let $z_i$ denote the $i$th iterate of Richardson iteration, then $z_{i}$ is computed as
\begin{equation}\label{eq:ric0}
\begin{aligned}
    z_{i} &= z_{i-1} + \omega (v-Az_{i-1})\\ & = (I-\omega A)z_{i-1} + \omega v,
\end{aligned}
\end{equation}
where $\omega \in \mathbb{R}$ is a scalar chosen so that the iterative procedure converges to $z$ \cite{richardson1911ix,varga1962iterative}. Following (\ref{eq:ric0}) and writing $z=(I-\omega A)z+\omega v$, we can express the approximation error at any 
iteration $m\in \mathbb{N}$ as $z_{m}-z = (I-\omega A)(z_{m-1}-z)$, from which it follows 
\begin{equation}\label{eq:ric1}
\begin{aligned}
    z_{m}-z = (I-\omega A)^{m}(z_{0}-z).
\end{aligned}
\end{equation}
Therefore, the norm of the absolute error satisfies the following estimate 
\begin{equation*}
    \norm{z_{m}-z} \leq \|I-\omega A\|^m\norm{z_0-z}.
\end{equation*}
The approximate solution produced by Richardson iteration is guaranteed to converge when $\rho\left(I - \omega A\right) < 1$. 
For Symmetric Positive Definite (SPD) matrices, the optimal value of $\omega$ is equal to $\omega_{\rm CR} = \frac{2}{\lambda_1 + \lambda_N}$, where $\lambda_1$ and $\lambda_N$ denote the smallest and largest eigenvalue of $A$, respectively. In this case, the spectral radius of the iteration matrix is equal to $\rho\left(I - \omega A\right) = 1-\dfrac{2}{1 + \kappa(A)}$ where $\kappa(A) = \lambda_N/\lambda_1$ \cite{saad2003iterative}.

\subsection{A straggler-tolerant scheme} 

We now turn our focus in environments with straggling workers. Let $m$ 
denote once again the number of Richardson iterations performed, and let 
$T_1,\ldots,T_m$ denote $m$ instances of the random variable $T$ with corresponding 
row subset samples ${\cal T}_1,\ldots,{\cal T}_m \subseteq \{1,2,\ldots,N\}$ of the 
random row subset ${\cal T}$ such that $T_i\equiv |{\cal T}_i|,\ i=1,\ldots,m$. 
In this case, we propose the following Richardson update scheme:
\begin{equation}\label{eq:ric2}
    \widehat{z}_{i} = (I-\widehat{\omega} D_{{\cal T}_i}A)\widehat{z}_{i-1} +\omega v, 
\end{equation} 
where $\omega \in \mathbb{R}$ is the scalar parameter associated with a convergent 
classical Richardson iteration and $\widehat{\omega} \in \mathbb{R}$. 
Note that when $\omega\neq \hat{\omega}$, the solution of the  linear system 
$Az=v$ does not equal the fixed point of (\ref{eq:ric2}) even when $T=N$. 

The formula in (\ref{eq:ric2}) is similar to that of the classical Richardson iteration 
(\ref{eq:ric0}) except that the $i$th iteration replaces the (constant) matrix 
$I-\omega A$ with the matrix $I-\widehat{\omega} D_{{\cal T}_i}A$. It is easy to see that 
for any $|{\cal T}_i| \neq N$ the matrix $\widehat{\omega} D_{{\cal T}_i}A$ is rank deficient 
and the matrix {$I-\widehat{\omega} D_{{\cal T}_i}A$ has $N-T_i$} eigenvalues of 
modulus one. Notice that the likelihood of the same $T$ and ${\cal T}$ will be sampled 
is low, especially as $m$ increases. This is because the probability that the row subsets 
${\cal T}_i,\ i=1,\ldots,m$, index the same row subset  $m$ consecutive times is 
${\mybinom{N}{T}}^{-m}$. Even when $T=N-1$, this probability is still ${\mybinom{N}{N-1}}^{-m}=N^{-m}$. 

Following the above discussion, the main question is whether the update formula (\ref{eq:ric2}) produces a sequence that converges to the solution of $Az=v$. Due to randomization, such convergence (if it occurs) will hold only in expectation, i.e., $\lim_{m\rightarrow \infty}\mathbb{E}\left[\widehat{z}_{m}-z\right]=0$. Intuitively, for a fixed $\widehat{\omega}$, we expect the sequence produced by (\ref{eq:ric2}) to make more progress toward the solution $z$ when $\mathbb{E}[T]$ is higher, i.e. when a larger number of rows of $A$ is sampled per matrix-vector product. 

\begin{algorithm}
{Straggler-tolerant} Richardson iteration for solving $Az=v$.
\label{alg:richardson}
\begin{algorithmic}[1]
    \STATE {\bf Given}: $A\in \mathbb{R}^{N\times N};\ 
    v\in \mathbb{R}^N;\ m\in \mathbb{N},\ \widehat{z}_0\in \mathbb{R}^N,\ \widehat{\omega},\ \omega \in \mathbb{R}$.
    \FOR {$i=1$ \TO $m$}
        \STATE Sample $T_i$ and ${\cal T}_i$
        \STATE $\widehat{z}_{i} = (I-\widehat{\omega} D_{{\cal T}_i}A)\widehat{z}_{i-1} + \omega v$ 
    \ENDFOR
    \RETURN $\widehat{z}_{m} $
\end{algorithmic}
\end{algorithm}

Algorithm \ref{alg:richardson} summarizes {straggler-tolerant} Richardson iteration. 
At iteration $i$, Algorithm \ref{alg:richardson} samples $T_i$ and ${\cal T}_i$ (in this order) and 
updates $\widehat{z}_{i-1}$ to $\widehat{z}_i$. Algorithm \ref{alg:richardson} returns once the 
user-chosen number of $m\in \mathbb{N}$ iterations is applied.

\begin{remark}
Algorithm \ref{alg:richardson} assumes a uniform random model to perform the matrix-vector product $\widehat{\omega}D_{{\cal T}_i}A\widehat{z}_{i-1}$. The approach outlined in this paper is different from the random
sparsified Richardson presented in \cite{weare2023randomly} where the update takes the form $z_{i} = 
(I-\omega A)\Phi_i(z_{i-1}) + v$ for some random sparsification operator $\Phi_i$ which requires an integer parameter of the maximum modulus retained values as well as a $N$-length vector of selection probabilities.
\end{remark}

\subsection{Convergence Analysis} \label{coan}

In this section, we discuss theoretical aspects of Algorithm \ref{alg:richardson} 
for the uniform distribution case. Our analysis focuses on the matrix expectations 
that result from the model in (\ref{eq:mvm0}) and is quite different from 
the analysis in the case of a simplified asynchronous Jacobi stationary iteration 
\cite{wolfson2019modeling}.

Starting with the algebraic manipulation
\begin{equation*}
    I-\widehat{\omega} D_{{\cal T}_i}A = I-\widehat{\omega} A + \widehat{\omega}(I- D_{{\cal T}_i})A,
\end{equation*}
it follows that (\ref{eq:ric2}) can be re-written as 
\begin{align*}
    \widehat{z}_{i+1} & = 
    \left(I-\widehat{\omega} A +\widehat{\omega} \left[I-D_{{\cal T}_i}\right]A\right)\widehat{z}_i+ \omega v\\
    & = \left(I-\widehat{\omega} A+E_i\right)\widehat{z}_i +  \omega v,
\end{align*}
where  
\begin{equation*}
    E_i = \widehat{\omega}(I-D_{{\cal T}_i})A.
\end{equation*}
The above equation implies that {when $\widehat{\omega}=\omega$} the $i$th iteration of 
{straggler-tolerant} Richardson iteration can be seen 
as a variation of the $i$th iteration of classical Richardson iteration 
in which the iteration matrix $I-\widehat{\omega} A$ is perturbed by the matrix 
$E_i$. The latter matrix can be understood as a sample of the 
random matrix $E_{\cal T} = \widehat{\omega}(I-D_{{\cal T}})A$. 

\begin{lemma} \label{lem1}
Let ${\cal T}$ denote a random subset of $\{1,2,\ldots,N\}$ whose cardinality 
depends on the random integer variable $T$ that takes values from $1$ to $N$, where ${\cal T}$ is sampled uniformly. 
Then, 
\begin{equation*}
        \mathbb{E}[E_{\cal T}] = \left(\dfrac{N-\mathbb{E}[T]}{N}\right) \widehat{\omega} A.
\end{equation*}
\end{lemma}
\begin{proof}
By the Law of Total Expectation \cite{weiss2012elementary}, the expectation $\mathbb{E}[E]$ can be written as $\mathbb{E}_T[\mathbb{E}_{\cal T}[E|T]]$ where 
the outer expectation is with respect to the cardinality $T$ of the random integer 
set ${\cal T}$ and the inner expectation is with respect to the content of ${\cal T}$. Denoting by $\mathbb{P}[E=E_{\cal T}|T]$ 
the probability that $E_{\cal T}$ is realized for a random row subset ${\cal T}$ of cardinality $T$, we have 
\begin{align*}
    \mathbb{E}_{\cal T}[E|T] & = \sum_{\cal T} \mathbb{P}[E=E_{\cal T}|T] E_{\cal T} \\
    & = \sum_{\cal T} \dbinom{N}{T}^{-1} \widehat{\omega}(I-D_{\cal T})A\\
    & = \widehat{\omega}\left(A - \dbinom{N}{T}^{-1} \dbinom{N-1}{T-1}A\right)\\
    & = \widehat{\omega}\left(A - \dfrac{T}{N}A\right).
\end{align*}
The proof follows by noticing  
$\mathbb{E}_T[\mathbb{E}_{\cal T}[E|T]]=\widehat{\omega}\mathbb{E}_T\left[\dfrac{N-T}{N}A\right] = 
\left(\dfrac{N-\mathbb{E}[T]}{N}\right)\widehat{\omega} A$.
\end{proof}
Lemma \ref{lem1} shows that the expectation of the perturbation introduced by 
Algorithm \ref{alg:richardson} is equal to a scalar multiple of the matrix 
$\widehat{\omega} A$. This multiple decreases as $\mathbb{E}[T]\rightarrow N$ and increases 
in the opposite direction. For example, $\mathbb{E}[E_{\cal T}]=\widehat{\omega} A/N$ when 
$\mathbb{E}[T]=N-1$ and 
$\mathbb{E}[E_{\cal T}] = (N-1)\widehat{\omega} A/N$  when $\mathbb{E}[T]=1$, respectively. Thus, increasing $\mathbb{E}[T]$, i.e., sampling a larger number of rows at each step, decreases the expectation of the matrix error. In the next proposition, we consider a special $\widehat{\omega}$ which yields an unbiased matrix-vector product associated with $I-\widehat{\omega}D_{\cal T}A$.

\begin{proposition} \label{pro121}
    Let $\widehat{\omega} = \frac{N}{\mathbb{E}[T]}\omega$, where $\omega \in \mathbb{R}$ is a scalar parameter. Then, 
    $\mathbb{E}\left[I-\widehat{\omega}D_{\cal T}A\right] = I-\omega A$.
\end{proposition}
\begin{proof}
    From Lemma \ref{lem1} we know that 
    $\mathbb{E}\left[D_{\cal T}A\right] = \dfrac{\mathbb{E}[T]}{N} A$, and thus 
    \begin{equation*}
        \mathbb{E}\left[I-\widehat{\omega}D_{\cal T}A\right] =I-\widehat{\omega}\frac{\mathbb{E}[T]}{N}A.
    \end{equation*}
    The proof follows by substituting $\widehat{\omega} = \frac{N}{\mathbb{E}[T]}\omega$.
\end{proof}

Proposition \ref{pro121} shows that if $\mathbb{E}[T]$ is known\footnote{Such information can become available either by the application itself or by warm-starting the hardware resources and executing several matrix-vector products before the actual solver is applied.} then we can pick $\widehat{\omega}$ so that in expectation the matrix to be multiplied at each iteration is equal to the deterministic matrix $I-\omega A$ used in classical Richardson iteration. 
This motivates us to consider whether adjusting the parameter $\widehat{\omega}$ can -in expectation- bring $\widehat{z}_m$ closer to $z_m$.  In the remaining of this section, we consider the following question: assume the classical Richardson iteration converges for some $\omega \in \mathbb{R}$, what is the sufficient condition for Algorithm \ref{alg:richardson} to converge in expectation.

Before we prove the convergence results of Algorithm \ref{alg:richardson}, we first derive the explicit expressions of $\hat{z}_m$ and $z_m$ via expanding the recursive update formulas of \eqref{eq:ric0} and \eqref{eq:ric2} in the next lemma.
\begin{lemma}\label{lem2}
    Let $I-\widehat{\omega} D_{{\cal T}_i}A = I-\widehat{\omega} A + E_i$. After $m$ iterations of Algorithm \ref{alg:richardson} and classical Richardson, we obtain
    \begin{equation*}
        \widehat{z}_m =   \prod\limits_{i=1}^{m}\left(I-\widehat{\omega} A+E_i\right) \widehat{z}_0 +\omega \left[\prod\limits_{i=2}^{m}\left(I-\widehat{\omega} A+E_i\right) +\cdots+ \left(I-\widehat{\omega} A+E_m\right)+I\right]v,
    \end{equation*}
    and 
        \begin{equation*}
        z_m = (I-\omega A)^m z_0 + \omega\left[(I-\omega A)^{m-1} +\cdots   +\omega(I-\omega A) + I\right] v,
    \end{equation*}
    respectively. 
\end{lemma}
\begin{proof}
The proof proceeds by induction. Using $\widehat{z}_{i}=(I-\widehat{\omega}D_{{\cal T}_i} A)\widehat{z}_{i-1} + \omega v$ and extending the first few $\widehat{z}_{i}$ terms, e.g., $\widehat{z}_{1}$, $\widehat{z}_{2}$, and $\widehat{z}_{3}$, yields:   
    \begin{align*}
        \widehat{z}_{1} & = (I-\widehat{\omega} A + E_1)\widehat{z}_{0} + \omega v\\
        \widehat{z}_{2} & = (I-\widehat{\omega} A + E_2)\left[(I-\widehat{\omega} A + E_1)\widehat{z}_{0} + \omega v\right] + \omega v\\ &  = \prod\limits_{i=1}^{2}\left(I-\widehat{\omega} A+E_i\right)\widehat{z}_{0} + \omega \left(I-\widehat{\omega} A+E_2\right)v + \omega v\\
        \widehat{z}_{3} & = (I-\widehat{\omega} A + E_3)\left[(I-\widehat{\omega}D_{{\cal T}_2} A)\widehat{z}_{1} + \omega v\right] + \omega v\\ &  = \prod\limits_{i=1}^{3}\left(I-\widehat{\omega} A+E_i\right)\widehat{z}_{0} + \omega \prod\limits_{i=2}^{3}\left(I-\widehat{\omega} A+E_i\right)v + \omega \left(I-\widehat{\omega} A+E_3\right)v + \omega v.
    \end{align*}
    Therefore, each new iteration of Algorithm \ref{alg:richardson} multiplies all previous terms by 
    a new matrix $I-\widehat{\omega} A+E_i$ and adds the term $\omega v$. The case for classical Richardson 
    is identical except that now we exploit the formula $z_{i}=(I-\omega A)z_{i-1} + \omega v$. 
\end{proof}

Finally, we prove the convergence result of Algorithm \ref{alg:richardson} in the next theorem.
\begin{theorem} \label{thm1}
    Let $\widehat{\omega} = \frac{N}{\mathbb{E}[T]}\omega$, where $\omega \in \mathbb{R}$ is the scalar parameter associated with a convergent classical Richardson iteration. Assume  $\widehat{z}_m$ and $z_m$ are the $m$th iterates generated by Algorithm \ref{alg:richardson} and classical Richardson iteration, respectively. If $\widehat{z}_0=z_0=f$ for some $f\in \mathbb{R}^N$, then
    \begin{equation*}
        \mathbb{E}\left[\widehat{z}_m\right]=z_m.
    \end{equation*}
\end{theorem}
\begin{proof}
    Following Lemma \ref{lem2}, we can write the difference $\widehat{z}_m-z_m$ as 
    \begin{align*}
       \widehat{z}_m - z_m\ =\ &\ \ \ \left[\prod\limits_{i=1}^{m}\left(I-\widehat{\omega} A+E_i\right) - (I-\omega A)^m\right] f\\
        & + \omega
        \left[\prod\limits_{i=2}^{m}\left(I-\widehat{\omega} A+E_i\right)-(I-\omega A)^{m-1}\right]v\\ 
        & + \cdots\\ 
        & + \omega
        \left[\prod\limits_{i=m-1}^{m}\left(I-\widehat{\omega} A+E_i\right)-(I-\omega A)^{2}\right]v\\ 
        & + \omega
        \left[\left(I-\widehat{\omega} A+E_m\right)-(I-\omega A)\right]v\\ 
        & + \omega v-\omega v.
    \end{align*}
    Taking expectations on both sides and applying the linearity of the expectation operator leads to
        \begin{align*}
       \mathbb{E}\left[\widehat{z}_m - z_m\right]\ =\ &\ \ \ 
       \left[\mathbb{E}\left[\prod\limits_{i=1}^{m}\left(I-\widehat{\omega} A+E_i\right)\right] - (I-\omega A)^m\right] f\\
        & + \omega
        \left[\mathbb{E}\left[\prod\limits_{i=2}^{m}\left(I-\widehat{\omega} A+E_i\right)\right]-(I-\omega A)^{m-1}\right]v\\ 
        & + \cdots\\ 
        & + \omega
        \left[\mathbb{E}\left[\prod\limits_{i=m-1}^{m}\left(I-\widehat{\omega} A+E_i\right)\right]-(I-\omega A)^{2}\right]v\\ 
        & + \omega
        \left[\mathbb{E}\left[I-\widehat{\omega} A+E_m\right]-(I-\omega A)\right]v.
    \end{align*}
    Therefore, if 
    $\mathbb{E}\left[\prod\limits_{i=m-k+1}^{m}\left(I-\widehat{\omega} A+E_i\right)\right] = (I-\omega A)^k,\ k=1,\ldots,m$, we can determine that all quantities inside the outermost brackets become equal and thus $\mathbb{E}\left[\widehat{z}_m-z_m\right] = 0$. To show the former, we need to determine $\mathbb{E}\left[\prod\limits_{i=m-k+1}^{m}\left(I-\widehat{\omega} A+E_i\right)\right]$. Since the matrices $E_i$ are independent and identically distributed, we can write 
    \begin{align*}
    \mathbb{E}\left[\prod\limits_{i=m-k+1}^{m}\left(I-\widehat{\omega} A+E_i\right)\right] & = 
    \mathbb{E}\left[\left(I-\widehat{\omega} A+E_{m}\right) \cdots \left(I-\widehat{\omega} A+E_{m-k+1}\right)\right]\\ 
    & = \mathbb{E}\left[\left(I-\widehat{\omega} A+E_{m}\right)\right]\cdots \mathbb{E}\left[\left(I-\widehat{\omega} A+E_{m-k+1}\right)\right]\\
    & =     \prod\limits_{i=m-k+1}^{m}\mathbb{E}\left[I-\widehat{\omega} A+E_i\right].
    \end{align*}
    Notice now that by Lemma \ref{lem1}, we have $\mathbb{E}[E_i] = \dfrac{N-\mathbb{E}[T]}{N}\widehat{\omega} A$. 
    Thus, 
\begin{align*}
    \mathbb{E}\left[I-\widehat{\omega} A+E_i\right] & = I-\widehat{\omega} A +\mathbb{E}[E_i]\\ & 
    =  I - \dfrac{\mathbb{E}[T]}{N}\widehat{\omega} A,
\end{align*}
which leads to 
\begin{equation*}
    \prod\limits_{i=m-k+1}^{m}\mathbb{E}\left[I-\widehat{\omega} A+E_i\right] = 
    \left(I - \dfrac{\mathbb{E}[T]}{N}\widehat{\omega} A\right)^k.
\end{equation*}
It follows that 
\begin{equation*}
    \mathbb{E}\left[\prod\limits_{i=m-k+1}^{m}\left(I-\widehat{\omega} A+E_i\right)\right] - (I-\omega A)^k = 
    \left(I - \dfrac{\mathbb{E}[T]}{N}\widehat{\omega} A\right)^k - (I-\omega A)^k,
\end{equation*}
which is equal to zero when $\widehat{\omega} = \dfrac{N}{\mathbb{E}[T]}\omega$.
\end{proof}
\begin{remark}\label{myrem}
Theorem \ref{thm1} implies that when $\widehat{\omega} = \omega$, $\mathbb{E}\left[\widehat{z}_m-z_m\right]$ has non-zero components along the direction of the vectors $\left[ \left(I - \dfrac{\mathbb{E}[T]}{N}\omega A\right)^k - (I-\omega A)^k\right]v,\ k=1,\ldots,m$, and thus $\mathbb{E}\left[\widehat{z}_m\right]$ does not converge to $z_m$.
\end{remark}

An immediate consequence of Theorem \ref{thm1} is that the expectation of the iterates generated by Algorithm \ref{alg:richardson} will converge to the true solution of $Az=v$.
\begin{corollary}\label{cor10}
Following the conditions in Theorem \ref{thm1}, we have
$$\lim\limits_{m\rightarrow \infty}\mathbb{E}\left[\widehat{z}_m\right]=z,$$ where 
$z$ is the solution of the linear system $Az=v$. 
\end{corollary}

\subsection{Effects of the iteration number $m$}

Theorem \ref{thm1} suggests that $\mathbb{E}[\widehat{z}_m]$ converges to the approximation $z_m$ returned by classical Richardson iteration. However, for $\widehat{z}_m$ to be a good approximation of the iterate $z_m$, the variance of each entry in $\widehat{z}_m$ should be small. While a general theoretical study of the variance lies beyond the scope of this paper, we will investigate the variance of $\widehat{z}_m$ as $m$ increases/decreases for the special case $\widehat{z}_0=\omega v$. 

Following Theorem \ref{thm1}, we can write the expectation of the approximation returned after $m$ steps of Algorithm \ref{alg:richardson} as $$\mathbb{E}[\widehat{z}_{m}] = \omega \sum\limits_{k=0}^m \left(I - \dfrac{\mathbb{E}[T]}{N}\widehat{\omega} A\right)^k v.$$ Since the matrix sum in the above equation is a geometric series, it can be further written as 
$$\sum\limits_{k=0}^m \left(I - \dfrac{\mathbb{E}[T]}{N}\widehat{\omega} A\right)^k = \left(\dfrac{\mathbb{E}[T]}{N}\widehat{\omega} A\right)^{-1}\left(I-\left(I-\dfrac{\mathbb{E}[T]}{N}\widehat{\omega} A\right)^{m+1}\right).$$
Thus, the expectation $\mathbb{E}[\widehat{z}_{m}]$ can be re-written as
\begin{equation}
    \mathbb{E}[\widehat{z}_m] = \omega \left(\dfrac{\mathbb{E}[T]}{N}\widehat{\omega} A\right)^{-1}\left(I-\left(I-\dfrac{\mathbb{E}[T]}{N}\widehat{\omega} A\right)^{m+1}\right) v.
    \label{eq:variance}
\end{equation}
It is easy to see that when $\widehat{\omega} = \dfrac{N}{\mathbb{E}[T]}\omega$ and 
$m$ is large enough, $\mathbb{E}[\widehat{z}_m]$ converges\footnote{Under the assumption that $\omega$ makes the spectral radius of $I-\omega A$ less than one.} to $A^{-1}v$:
\begin{equation}
    \mathbb{E}[\widehat{z}_m] \rightarrow z=A^{-1}v \  \ \text{as $m$ $\rightarrow$ $\infty$.}
    \label{eq:meanconverge}
\end{equation}

We now turn our attention to the covariance of 
$\widehat{z}_m$ as $m\rightarrow \infty$  \cite{walpole1993probability}:
\begin{equation}
\label{eq:convdecomp}
    \mathrm{Cov}(\widehat{z}_m) = \mathbb{E}[\widehat{z}_m \widehat{z}_m^{\top}] - \left( \mathbb{E}[\widehat{z}_m] \right) \left( \mathbb{E}[\widehat{z}_m] \right)^{\top}.
\end{equation}
Based on \eqref{eq:meanconverge}, we know that $\left( \mathbb{E}[\widehat{z}_m] \right) \left( \mathbb{E}[\widehat{z}_m] \right)^{\top} $ converges to $A^{-1}vv^{\top}A^{-\top}$ as $m\rightarrow \infty$.
Since $\left( \mathbb{E}[\widehat{z}_m] \right) \left( \mathbb{E}[\widehat{z}_m] \right)^{\top}$ converges to a constant matrix, we now focus on the first term on the right-hand side of \eqref{eq:convdecomp}. To this end, define $F_m = \sum_{j=1}^{m} \prod_{i=j}^{m} (I - \widehat{\omega} A + E_i)$. Based on Lemma \ref{lem2}, $\mathbb{E}[\widehat{z}_m \widehat{z}_m^{\top}]$ can be decomposed as 
\begin{equation}
\label{eq:covv}
\begin{aligned}
    \mathbb{E}[\widehat{z}_m \widehat{z}_m^{\top}] & = \omega^2 \mathbb{E}\left[(F_m+I)vv^{\top}(I+F_m^{\top})\right]\\
    & = \omega^2\left(\mathbb{E}\left[F_mvv^{\top}F_m^{\top}\right] + \mathbb{E}\left[vv^{\top}F_m^{\top}\right] + \mathbb{E}\left[F_mvv^{\top}\right] + \mathbb{E}\left[vv^{\top}\right]\right).
\end{aligned}
\end{equation}
Following the proof of Theorem \ref{thm1}, we know that $\mathbb{E}\left[I - \widehat{\omega} A + E_i\right]=
I-\dfrac{\mathbb{E}[T]}{N}\widehat{\omega} A$. Thus, we can write 
\begin{equation*}
    \mathbb{E}\left[F_m\right] = \sum_{j=1}^{m} 
\prod_{i=j}^{m} \mathbb{E}\left[I - \widehat{\omega} A + E_i\right] = \sum_{j=1}^{m} \left(I - \dfrac{\mathbb{E}[T]}{N}\widehat{\omega} A\right)^{j}.
\end{equation*}
Notice that when $\widehat{\omega} = \dfrac{N}{\mathbb{E}[T]}\omega$ and 
$\rho(I-\omega A)<1$, the expectation of the matrix $F_m$ is instead equal to   
$\mathbb{E}\left[F_m\right] = \sum_{j=1}^{m} \left(I - {\omega} A\right)^{j}$, 
which converges to $(\omega A)^{-1}-I$ as $m\rightarrow\infty$. 

The above analysis implies that the variation of $\mathrm{Cov}(\widehat{z}_m)$\footnote{Note that $v$ is constant and thus  $\mathbb{E}\left[F_mvv^{\top}\right]=\mathbb{E}\left[F_m\right]vv^{\top}$.} is mainly due to $\mathbb{E}\left[F_mvv^{\top}F_m^{\top}\right]$  as $m$ becomes large enough. We can compactly express $\mathbb{E}\left[F_mvv^{\top}F_m^{\top}\right]$ in the form of a double sum as follows:
\begin{equation*}
\mathbb{E}[F_mvv^{\top}F_m^{\top}] = \sum_{j=1}^{m} \sum_{k=1}^m 
\mathbb{E}\left[\left(\prod_{i=j}^{m} I - \widehat{\omega} A + E_i\right) vv^{\top} \left(\prod_{i=k}^{m} 
I - \widehat{\omega} A + E_i\right)^{\top}\right].
\end{equation*}
{Denote} now $\mathbb{E}\left[\left(\prod_{i=j}^{m} I - \widehat{\omega} A + E_i\right)vv^\top \left(\prod_{i=k}^{m} 
I - \widehat{\omega} A + E_i\right)^{\top}\right]$ by $E_{jk}$. Then, we can write
\begin{align*}
E_{jk} &
=\mathbb{E}\left[\left(\prod_{i=j}^{m} I - \widehat{\omega} A + E_i\right)v\right] \mathbb{E}\left[v^\top \left(\prod_{i=k}^{m} 
I - \widehat{\omega} A + E_i\right)^{\top}\right] + C_{jk}\\
& = \left(I - \frac{\mathbb{E}[T]}{N} \widehat{\omega} A\right)^{m-j+1}vv^{\top}
\left(I - \frac{\mathbb{E}[T]}{N} \widehat{\omega} A^\top\right)^{m-k+1} +C_{jk},
\end{align*} 
where $C_{jk}$ denotes the $N\times N$ covariance matrix between 
$\left(\prod_{i=j}^{m} I - \widehat{\omega} A + E_i\right)v$ and 
$\left(\prod_{i=k}^{m} I - \widehat{\omega} A + E_i\right)v$. While an extended 
analysis on the magnitude of the diagonal entries of $E_{jk}$ lies 
beyond the scope of the present paper, we note that when  
$\widehat{\omega}=\dfrac{N}{\mathbb{E}[T]}\omega$ and $\rho(I-\omega A)<1$, 
the matrix $\left(I - \frac{\mathbb{E}[T]}{N} \widehat{\omega} A\right)^{m-j+1}vv^{\top}\left(I - \frac{\mathbb{E}[T]}{N} \widehat{\omega} A^\top\right)^{m-k+1}$  
converges to zero as $m$ increases and $j,\ k$ remain fixed. This 
observation indicates that, outside of the influence of $C_{jk}$,  
the diagonal entries of $\mathbb{E}[F_mvv^{\top}F_m^{\top}]$ might 
increase marginally as $m$ increases beyond a certain value.

\begin{figure}
\centering
    \includegraphics[width=1.0\textwidth]{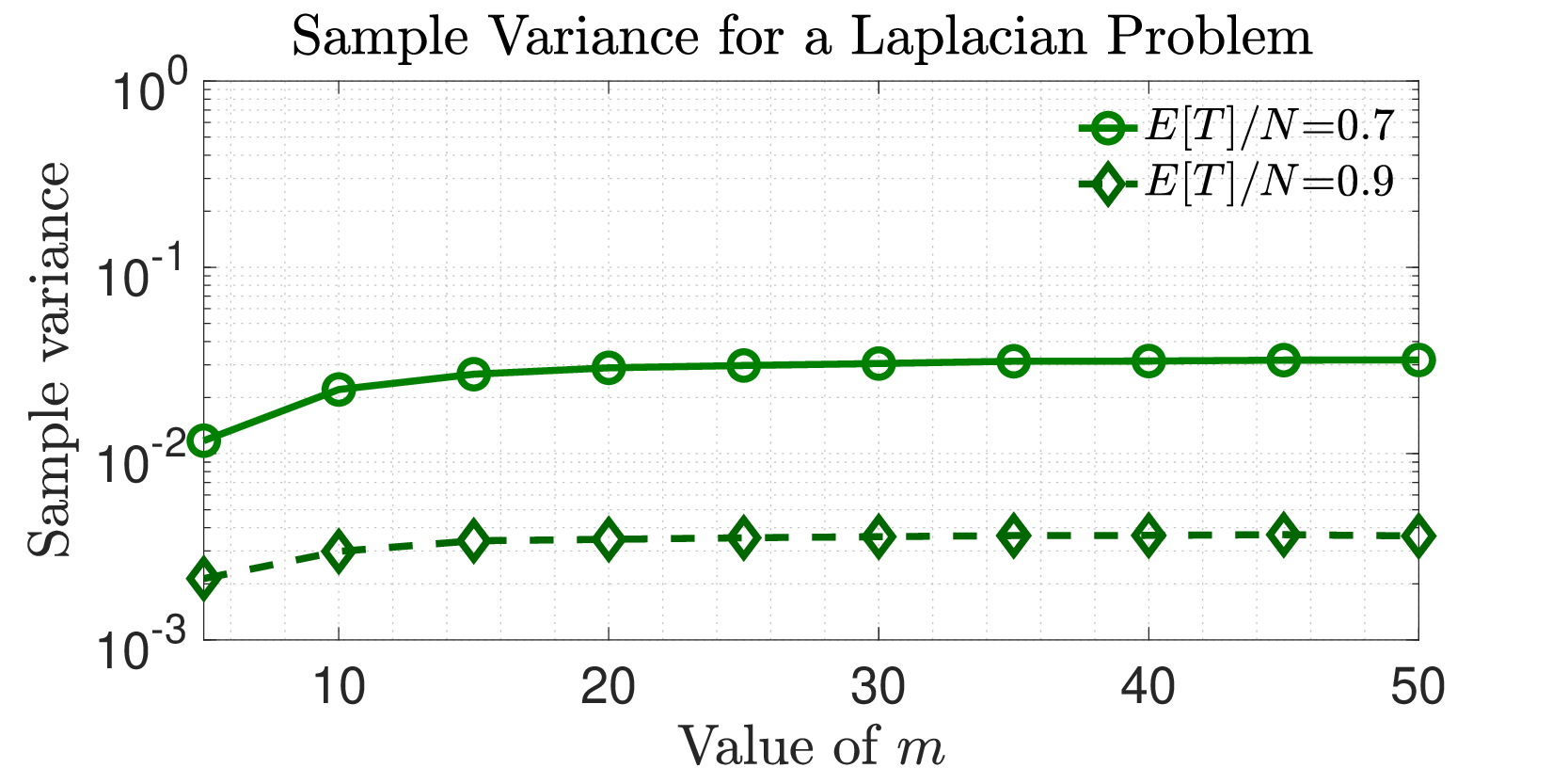}  \vspace*{-5mm}
    \caption{Average sample variance of the entries of $\widehat{z}_m$ for a $N=10^3$ model Laplacian problem.}\label{figm0}
\end{figure}
Figure \ref{figm0} plots the average sample variance of the entries of 
$\widehat{z}_m$ obtained via averaging over five hundred repetitions of 
Algorithm \ref{alg:richardson} for a matrix problem stemming from a 
Laplacian equation with Dirichlet boundary conditions discretized on a 
$10\times 10\times 10$ grid with a 7-pt stencil. The value of $m$ is varied 
from five to fifty, and we perform the same experiment for the values 
$\mathbb{E}[T]/N=0.7$ and $\mathbb{E}[T]/N=0.9$. The curves agree with our earlier 
discussion in which an increase in the value of $m$ generally leads to an increase 
in the magnitude of the diagonal entries of $\mathrm{Cov}(\widehat{z}_m)$ up to a certain value of $m$, beyond which the increase is marginal and essentially plateaus. 

Finally, we note that the expression in \eqref{eq:covv} suggests that $\mathbb{E}[\widehat{z}_m \widehat{z}_m^{\top}]$ depends quadratically on $\omega$ and thus 
smaller values of $\omega$ might reduce the overall variance of $\widehat{z}_m$. 
Note that reducing $\omega$ slows down the convergence of classical 
Richardson iteration and the sample mean of $\widehat{z}_m$ might become an inferior approximation of 
$z$. We will demonstrate these trade-offs in greater detail in Section \ref{sec5}.

\section{Chebyshev semi-iterative method with straggling workers}\label{sec:cheby}

In this section, we consider extending Algorithm \ref{alg:richardson} to the 
stationary Chebyshev semi-iterative iteration with fixed coefficients 
\cite{golub1961chebyshev,gutknecht2002chebyshev} for solving $Az=v$ when $A$ 
is SPD. Assume the eigenvalues of $A$ are within the interval $[\alpha,\beta]$.  
Then, the stationary Chebyshev semi-iterative method generates the new iterate 
$z_m$ based on the past two iterates $z_{m-1}$ and $z_{m-2}$ via the formula 
\begin{equation} \label{eq:cheb}
    z_m = z_{m-1} + \eta(z_{m-1}-z_{m-2}) + \nu(v-Az_{m-1}),
\end{equation}
where $\eta,\ \nu \in \mathbb{R}$ are two fixed 
scalars defined as in \cite{saad2003iterative}:
\begin{equation}\label{eq:coeffs}
\eta = \rho^2, \quad \nu = \frac{2\rho}{\delta},
\end{equation}
and 
\[
\rho = \frac{\alpha+\beta}{\beta-\alpha} -\sqrt{\left(\frac{\alpha+\beta}{\beta-\alpha}\right)^2-1} \ \text{and} \ \delta = \frac{\alpha+\beta}{2}.
\]
The iteration described above is commonly referred to as ``second-order 
iteration". A significant advantage of the Chebyshev 
iteration is that it does not require any inner products, making it 
particularly well-suited for distributed computing environments. The stationary iteration defined by (\ref{eq:cheb}) can also converge faster 
than the standard Richardson iteration applied to $Az=v$. 

The iteration \eqref{eq:cheb} can be recast in a $2\times 2$ block form:
\begin{equation} \label{cheb3}
    \begin{aligned}
    \underbrace{\begin{bmatrix}
    z_{m}\\
    z_{m-1}
    \end{bmatrix}}_{d_m}
    & =
    \underbrace{\begin{bmatrix}
    (1+\eta)I- \nu A& -\eta I\\
    I&0\\
    \end{bmatrix}}_{C}
    \underbrace{\begin{bmatrix}
    z_{m-1}\\
    z_{m-2}
    \end{bmatrix}}_{d_{m-1}}
    +
    \underbrace{\nu\begin{bmatrix}
     v\\
    0
    \end{bmatrix}}_{b}.
\end{aligned}
\end{equation}
This system is of size $2N \times 2N$ and thus is twice as big as the original 
linear system $Az=v$. Moreover, even when $A$ is symmetric, the coefficient 
matrix ${C}$ in (\ref{cheb2}) is non-symmetric. The form \eqref{cheb3} will 
merely be used for analyzing the convergence.

To derive the {straggler-tolerant} version of \eqref{eq:cheb}, 
we first rearrange the terms on the right-hand side of \eqref{eq:cheb} in the 
following way:
\begin{equation}\label{cheb1}
\begin{aligned} 
    z_m  = z_{m-1} + \eta(z_{m-1}-z_{m-2})+ \nu v - \nu Az_{m-1}.
\end{aligned}
\end{equation}
Notice that only the last term on the right-hand side of the above equation 
involves the matrix-vector product associated with the $N\times N$ matrix $A$. 
Based on our discussion in the previous section, we propose the following 
Chebyshev iteration scheme:
\begin{equation}
\label{eq:randomcheb}
    \widehat{z}_m = \widehat{z}_{m-1} + \eta(\widehat{z}_{m-1}-\widehat{z}_{m-2})+ \nu v - \widehat{\nu} {\cal D}_{{\cal T}_m}A\widehat{z}_{m-1},
\end{equation}
which stems from replacing the classical matrix-vector product with the 
model defined in (\ref{eq:mvm0}) for a random row subset ${\cal D}_{{\cal T}_m}$, 
and  $\widehat{\nu}=\frac{N}{\mathbb{E}[T]}\nu\in \mathbb{R}$. Note that 
the expression of $\widehat{z}_m$ now involves the scalars $\{\eta,\nu,\widehat{\nu}\}$
instead of $\{\eta,\nu\}$. The corresponding $2\times 2$ block form is as follows:

\begin{equation} \label{cheb2}
\begin{aligned}
    \underbrace{\begin{bmatrix}
    \widehat{z}_{m}\\
    \widehat{z}_{m-1}
    \end{bmatrix}}_{\widehat{d}_m}
    & =
    \underbrace{\begin{bmatrix}
    (1+\eta)I-\widehat{\nu} {\cal D}_{{\cal T}_m}A& -\eta I\\
    I&0\\
    \end{bmatrix}}_{\widehat{C}_m}
    \underbrace{\begin{bmatrix}
    \widehat{z}_{m-1}\\
    \widehat{z}_{m-2}
    \end{bmatrix}}_{\widehat{d}_{m-1}}
    +
    \underbrace{\nu\begin{bmatrix}
     v\\
    0
    \end{bmatrix}}_{b}.
\end{aligned}
\end{equation}
Similar to the case of Richardson iteration, our goal is to show that the iterate $\widehat{d}_m$ 
associated with the extended $2\times 2$ block system in (\ref{cheb2}) converges -in expectation- to 
the same stationary point that the classical $2\times 2$ block form extension 
in (\ref{cheb1}) does. 

\begin{theorem}\label{thm2}
    Assume $A$ is SPD and $\eta, \nu$ are set based on \eqref{eq:coeffs}. Then, we have
       \begin{equation*}
        d_m = C^m d_0 + \left[C^{m-1} +\cdots +C + I\right] b,
    \end{equation*}
    and 
        \begin{equation*}
        \widehat{d}_m =   \prod\limits_{i=1}^{m} \widehat{C}_i \widehat{d}_0 +\left[\prod\limits_{i=2}^{m}\widehat{C}_i +\cdots+ \widehat{C}_m+I\right]b,
    \end{equation*}
    for the iterations \eqref{cheb3} and \eqref{cheb2}, respectively.
    Moreover, if $\widehat{d}_0=d_0=f$ for some $f\in 
    \mathbb{R}^{2N}$, then
    \begin{equation*}
        \mathbb{E}\left[\widehat{d}_m\right]=d_m.
    \end{equation*}
\end{theorem}
\begin{proof}
The proof follows the same logic as in Theorem \ref{thm1}. The only additional 
consideration is the computation of the expectation 
$\mathbb{E}\left[\prod\limits_{i=j}^{m}\widehat{C}_i\right]$. Due to 
the independence of each sample, once again we can write 
\begin{equation*}
    \mathbb{E}\left[\prod\limits_{i=j}^{m}\widehat{C}_i\right] = 
    \mathbb{E}\left[\widehat{C}_j\cdots \widehat{C}_m\right]
    = \prod\limits_{i=j}^{m} \mathbb{E}\left[\widehat{C}_i\right].
\end{equation*}
Each term $\mathbb{E}\left[\widehat{C}_i\right]$ can be then written as
\begin{align*}
    \mathbb{E}\left[\widehat{C}_i\right] &= 
    \mathbb{E}\Biggl[\begin{bmatrix}
    (1+\eta)I-\widehat{\nu} {\cal D}_{{\cal T}_{i}}A& -\eta I\\
    I&0\\
    \end{bmatrix}\Biggr]\\
    &= 
    \begin{bmatrix}
    \mathbb{E}\left[(1+\eta)I-\widehat{\nu} {\cal D}_{{\cal T}_i}A\right]& -\mathbb{E}\left[\eta I\right]\\
    \mathbb{E}\left[I\right]&\mathbb{E}\left[0\right]\\
    \end{bmatrix}\\
   &= \begin{bmatrix}
    (1+\eta)I-\widehat{\nu}\mathbb{E}\left[ {\cal D}_{{\cal T}_i}A\right]& -\eta I\\
    I&0\\
    \end{bmatrix}.
\end{align*}
From Proposition \ref{pro121}, we know that $\mathbb{E}\left[ {\cal D}_{{\cal T}_i}A\right]=\frac{\mathbb{E}[T]}{N}A$. Thus, $\mathbb{E}\left[\widehat{C}_i\right]=C$, which concludes the first part. The second part follows directly from Theorem \ref{thm1}.
\end{proof}

\section{Numerical Experiments} \label{sec5}

In this section, we illustrate the numerical performance of {straggler-tolerant} 
Richardson iteration (Algorithm \ref{alg:richardson}) and {straggler-tolerant} 
Chebyshev iteration via the update formula in \eqref{eq:randomcheb}. Our numerical experiments are conducted in a Matlab environment (version R2023b), using 64-bit 
arithmetic, on a single core of a computing system equipped with an 
Apple M1 Max processor and 64 GB of system memory. Unless mentioned 
otherwise, the right-hand side of each problem will be set equal to the 
matrix-vector product between the matrix $A$ and the vector of all ones 
and the initial approximation is set as a zero vector for all problems. 

Throughout our experiments, we sample the number $T$ of observed entries 
per matrix-vector product (\ref{eq:mvm0}) uniformly from the interval 
$\left[\mathbb{E}[T]-100,\mathbb{E}[T]+100\right]$. We represent the ratio 
of the expectation of $T$ over the problem dimension $N$ by the scalar 
$\tau \in \mathbb{R}$:
\begin{equation*}
    \tau = \dfrac{\mathbb{E}[T]}{N},
\end{equation*}
and the optimal scalar parameter in classical Richardson iteration by the 
scalar $\omega_{\rm CR}$:
\begin{equation*}
    \omega_{\rm CR}=\frac{2}{\lambda_1+\lambda_N}.
\end{equation*}

Let $z_m$ denote the iterate returned by performing
$m$ steps of classical Richardson iteration to solve the
linear system $Az=v$ and $\widehat{z}_m^{(i)}$
denote the output during the $i$th execution of  Algorithm 
\ref{alg:richardson} with $m$ iterations. Our experimental results focus primarily on the approximation of the 
following two quantities:
\begin{itemize}
\item  Approximation error of $\frac{1}{L}\sum_{i=1}^L\widehat{z}_m^{(i)}$ to $z_m$, which is measured by the Mean Squared Error (MSE) of the vector $\frac{1}{L}\sum_{i=1}^L\widehat{z}_m^{(i)}-z_m$.
\item Approximation error of 
$\frac{1}{L}\sum_{i=1}^L\widehat{z}_m^{(i)}$ to $z$, where is measured by the MSE of $\frac{1}{L}\sum_{i=1}^L\widehat{z}_m^{(i)}-z$.
\end{itemize}

\subsection{Illustration of Algorithm \ref{alg:richardson}} \label{secnum1}

In this section, we consider the MSE achieved by classical and {straggler-tolerant} 
Richardson iteration on a set of sparse and model problems. For classical Richardson iteration we only consider the optimal parameter value $\omega_{\rm CR}$. For Algorithm \ref{alg:richardson} we consider both $\widehat{\omega} = \frac{N}{\mathbb{E}[T]}\omega_{\rm CR}$ and 
${\omega}=\omega_{\rm CR}$ as the scalar parameter associated with the matrix-vector products.

\subsubsection{Performance on general sparse matrices}

Figure \ref{fig1} plots the MSE of the error vector 
$\frac{1}{L}\sum_{i=1}^L\widehat{z}_m^{(i)}-z_m$ 
when the scalar parameter associated with matrix-vector products is set as $\omega$ (dashed line) and $\widehat{\omega}$ (dashed-dotted line) as 
well as the MSE of the error vector $\frac{1}{L}\sum_{i=1}^L\widehat{z}_m^{(i)}-z$ (solid line). As the test matrix, we choose the \textsc{crystm01} sparse matrix from SuiteSparse\footnote{\url{https://sparse.tamu.edu/}} \cite{davis2011university} matrix collection. This matrix has size $N=4,875$ and 105,339 non-zero entries. We perform experiments for the values $\tau \in\{0.6,0.75,0.9\}$ and $m\in \{20,50\}$. 
Recall now that the MSE of $\frac{1}{L}\sum_{i=1}^L\widehat{z}_m^{(i)}-z_m$ measures how well the mean of the samples $\widehat{z}_m^{(i)}$ produced by Algorithm \ref{alg:richardson} 
approximates the vector $z_m$. Therefore, increasing the number of samples can close the gap between $\frac{1}{L}\sum_{i=1}^L\widehat{z}_m^{(i)}$ and $z_m$ as indicated by the decrease of the red dashed-dotted line as we move rightwards on the real axis regardless of the value of $m$. 
On the other hand, reducing the MSE of $\frac{1}{L}\sum_{i=1}^L\widehat{z}_m^{(i)}-z$ generally requires that $z_m$ is also a good approximation of $z$. This is the reason why for $m=20$ we see the approximation to $z_m$ by $\frac{1}{L}\sum_{i=1}^L\widehat{z}_m^{(i)}$ 
improves as a function of the total number of samples $L$ while the approximation to $z$ decreases 
very slowly. On the other hand, when $m=50$, $z_m$ is a more accurate approximation to $z$ and $\frac{1}{L}\sum_{i=1}^L\widehat{z}_m^{(i)}$ converges towards this more accurate approximation. 
As a result, the MSE of $\frac{1}{L}\sum_{i=1}^L\widehat{z}_m^{(i)}-z$ reduces at about the same 
rate as the MSE of $\frac{1}{L}\sum_{i=1}^L\widehat{z}_m^{(i)}-z_m$. It is important to notice that 
the above behavior holds only when the scalar parameter is equal to $\widehat{\omega}$. 
If instead, one replaces $\widehat{\omega}$ by $\omega_{\rm CR}$, the MSE essentially stagnates due to unresolved residuals as indicated in Remark \ref{myrem}.

Figure \ref{fig2} plots the same results for the sparse matrix \textsc{bundle1} with 
size $N=10,581$ and 770,811 non-zero entries. Similarly, increasing 
the value of either $\tau$ or $L$ reduces the MSE $\frac{1}{L}\sum_{i=1}^L\widehat{z}_m^{(i)}-z_m$ if 
the scalar parameter is equal to $\widehat{\omega}$. Moreover, again in agreement 
with the above results, increasing the value of $m$ makes the MSE of 
$\frac{1}{L}\sum_{i=1}^L\widehat{z}_m^{(i)}-z_m$ essentially track that of $\frac{1}{L}\sum_{i=1}^L\widehat{z}_m^{(i)}-z$. 

\begin{figure}
\centering
    \includegraphics[width=0.32\textwidth]{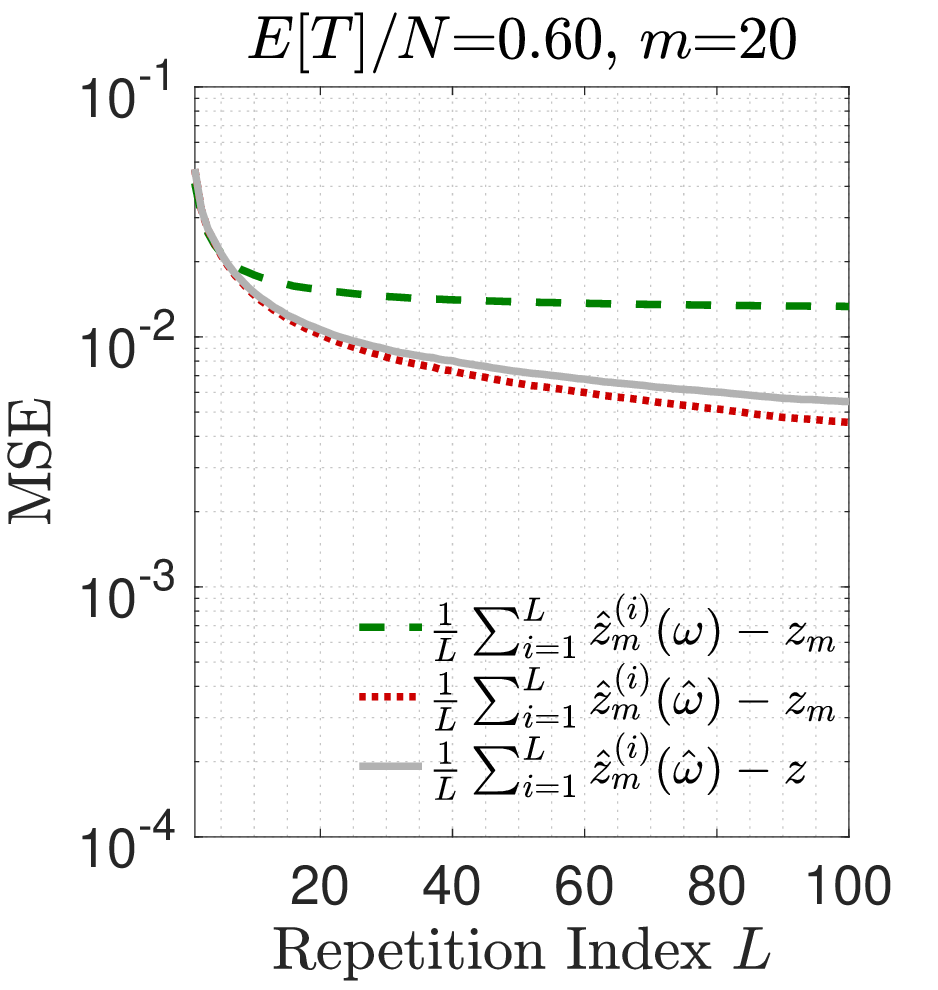}   
    \includegraphics[width=0.32\textwidth]{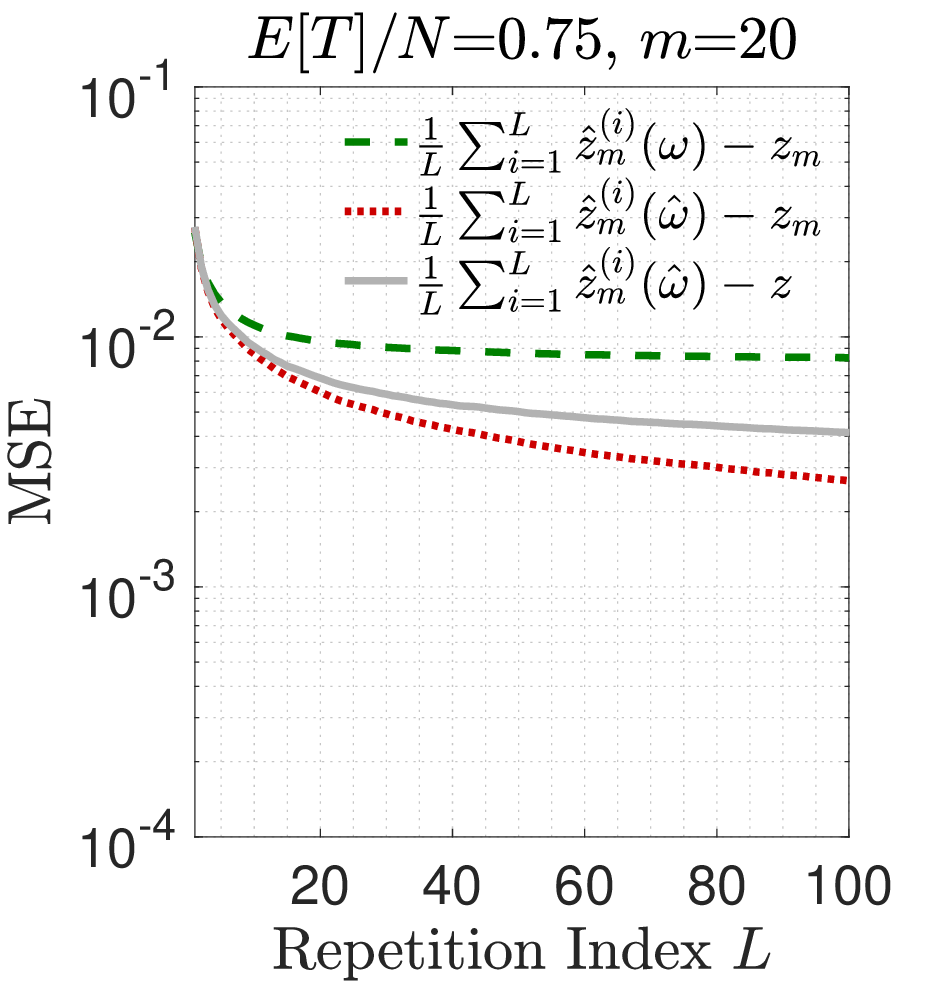}   
    \includegraphics[width=0.32\textwidth]{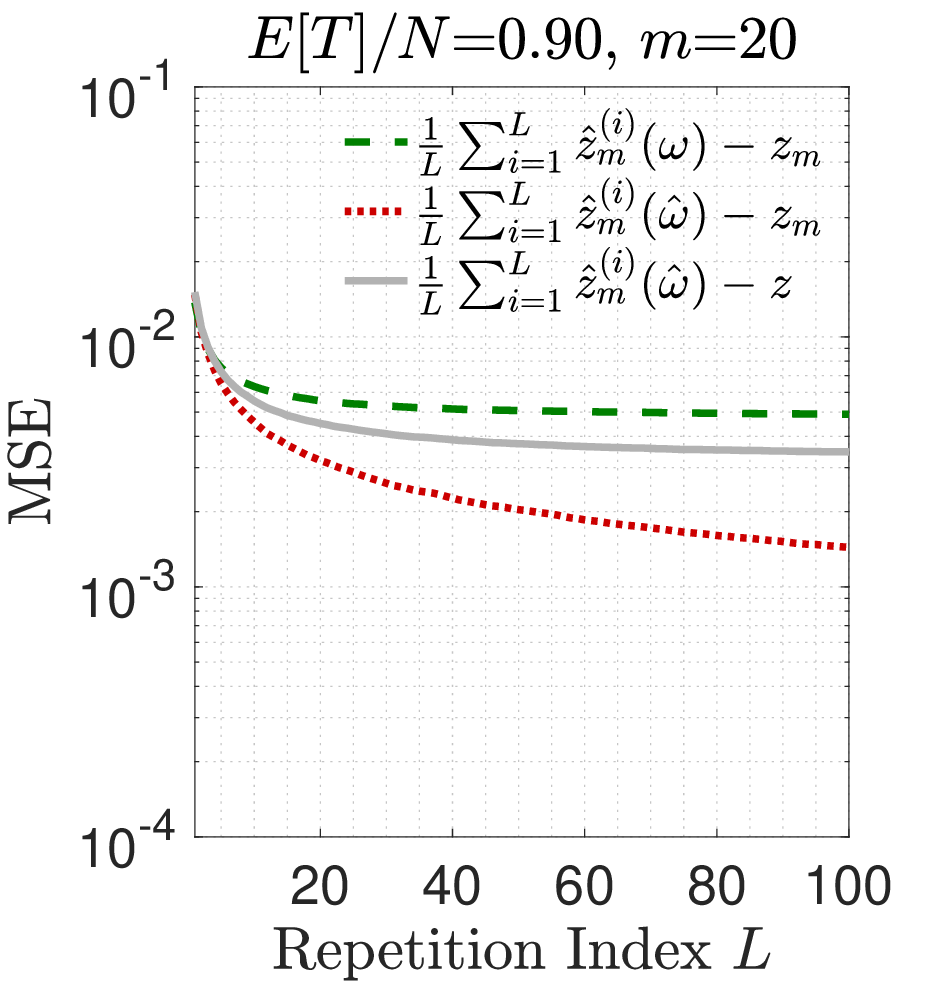}   
    \includegraphics[width=0.32\textwidth]{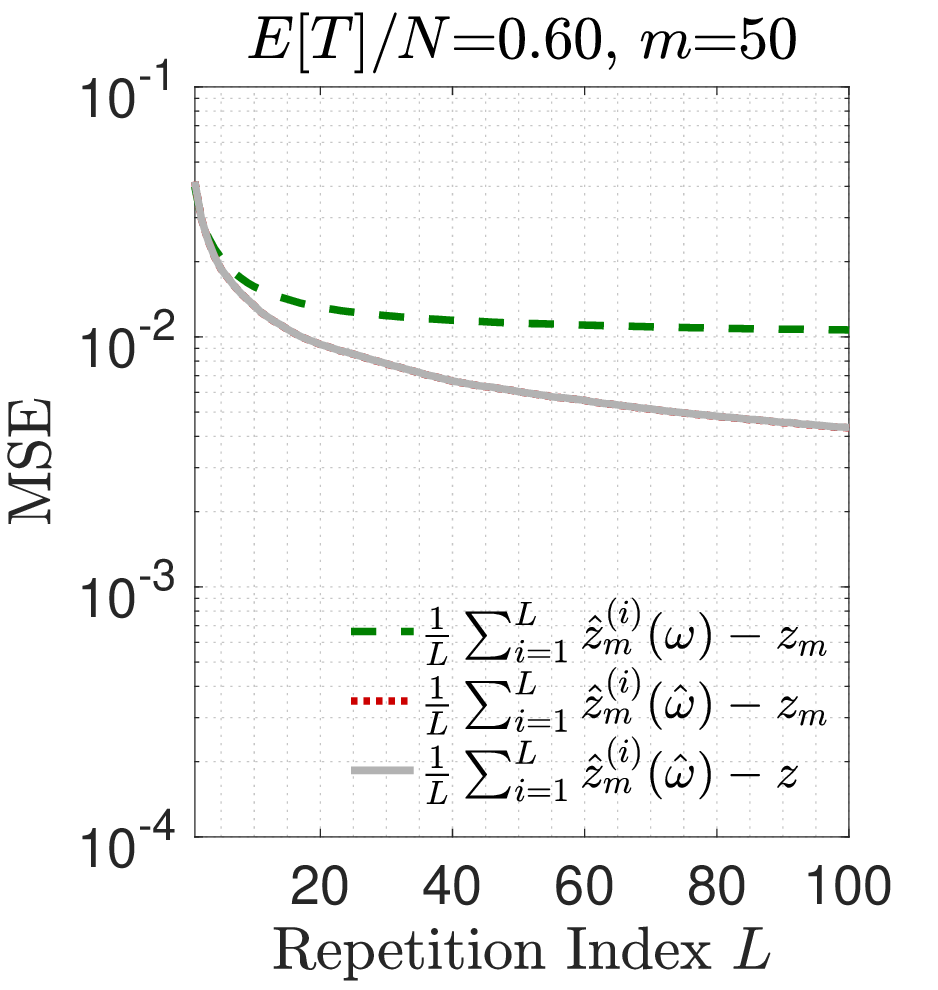}   
    \includegraphics[width=0.32\textwidth]{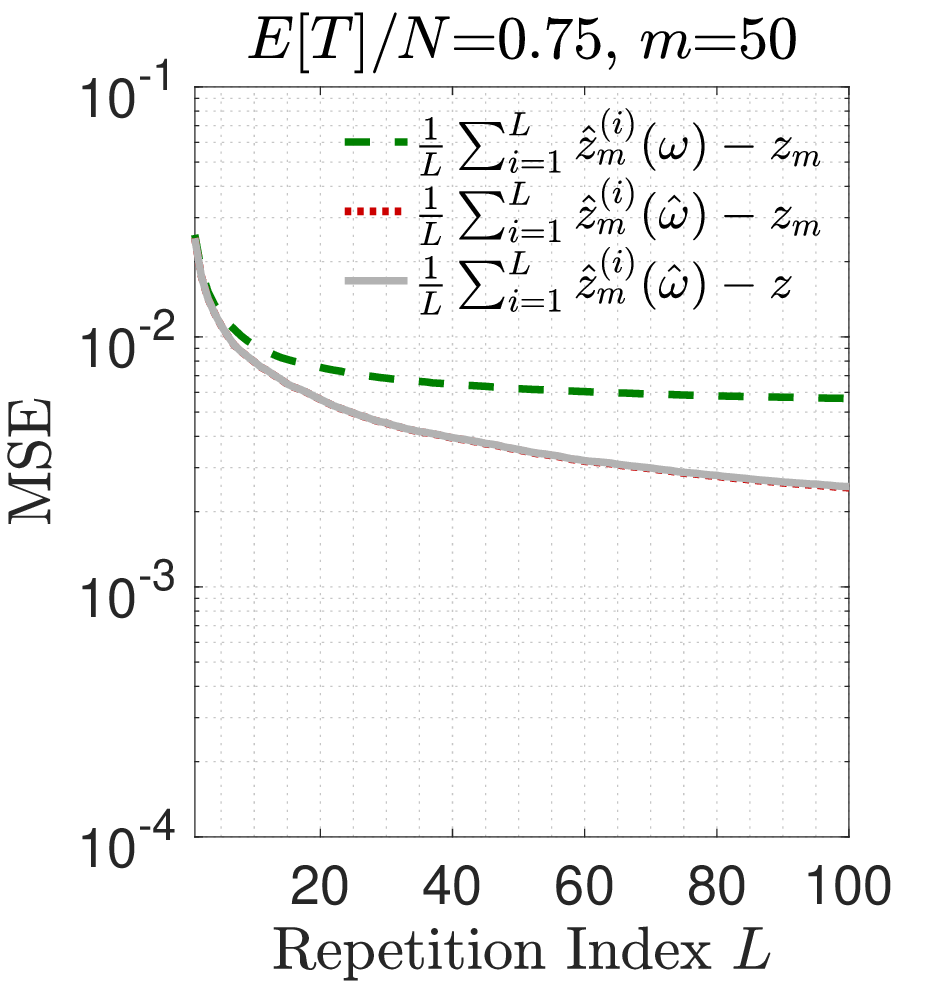}   
    \includegraphics[width=0.32\textwidth]{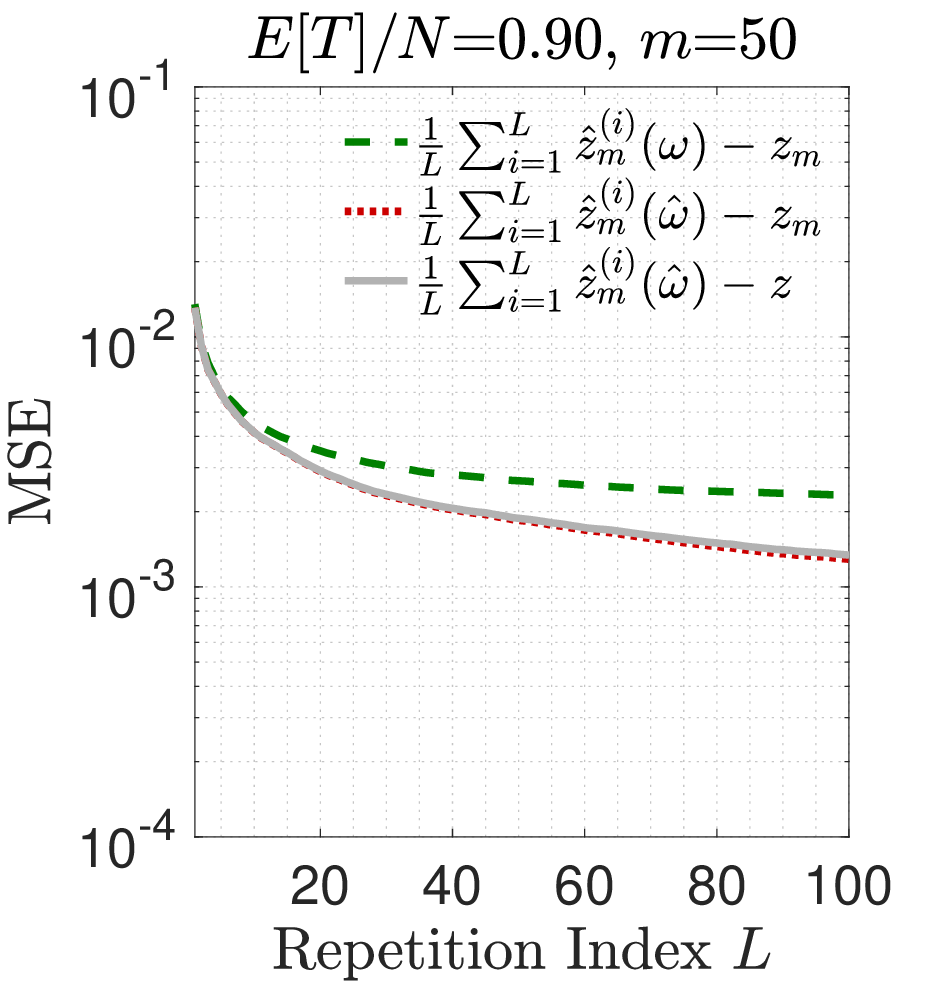}   
    \caption{MSE of $\frac{1}{L}\sum_{i=1}^L\widehat{z}_m^{(i)}-z_m$ on up to $L=100$ trials of Algorithm \ref{alg:richardson} for the matrix \textsc{crystm01} as well as   MSE of the quantity $\frac{1}{L}\sum_{i=1}^L\widehat{z}_m^{(i)}-z$ where $z$ is the solution of $Az=v$.}\label{fig1}
\end{figure}

\begin{figure}
\centering
    \includegraphics[width=0.32\textwidth]{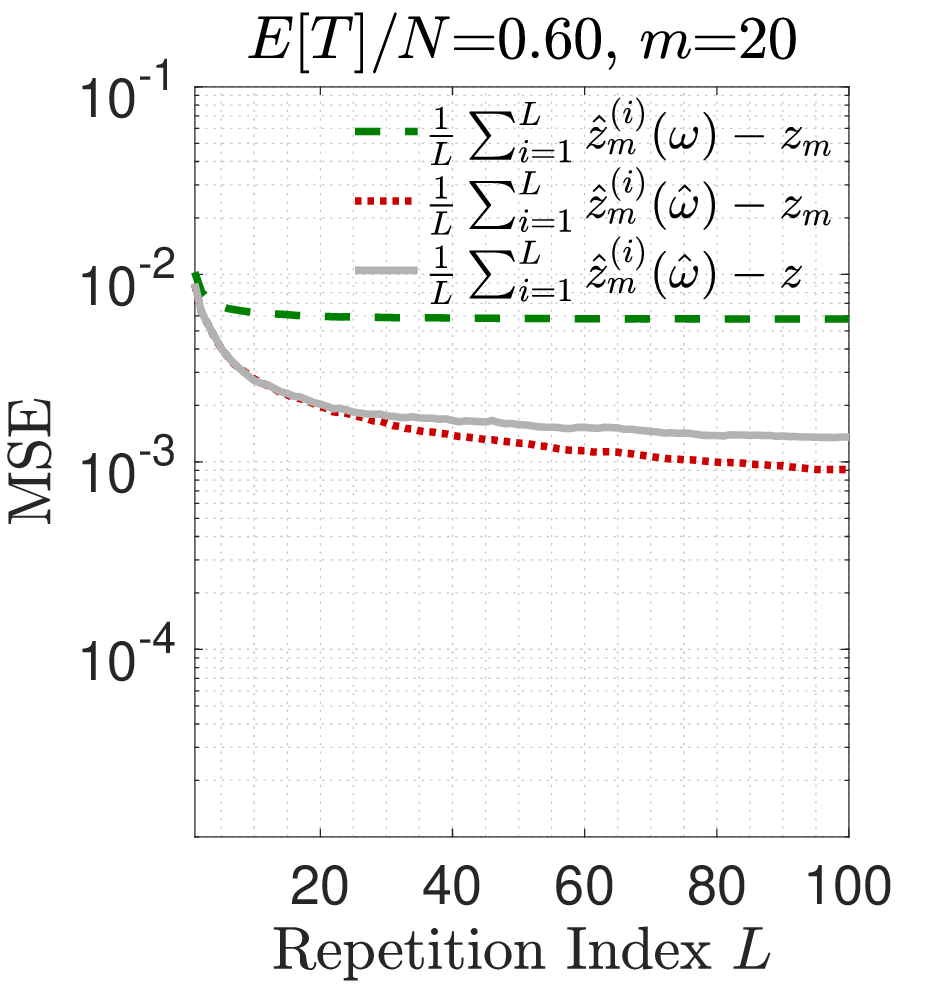}   
    \includegraphics[width=0.32\textwidth]{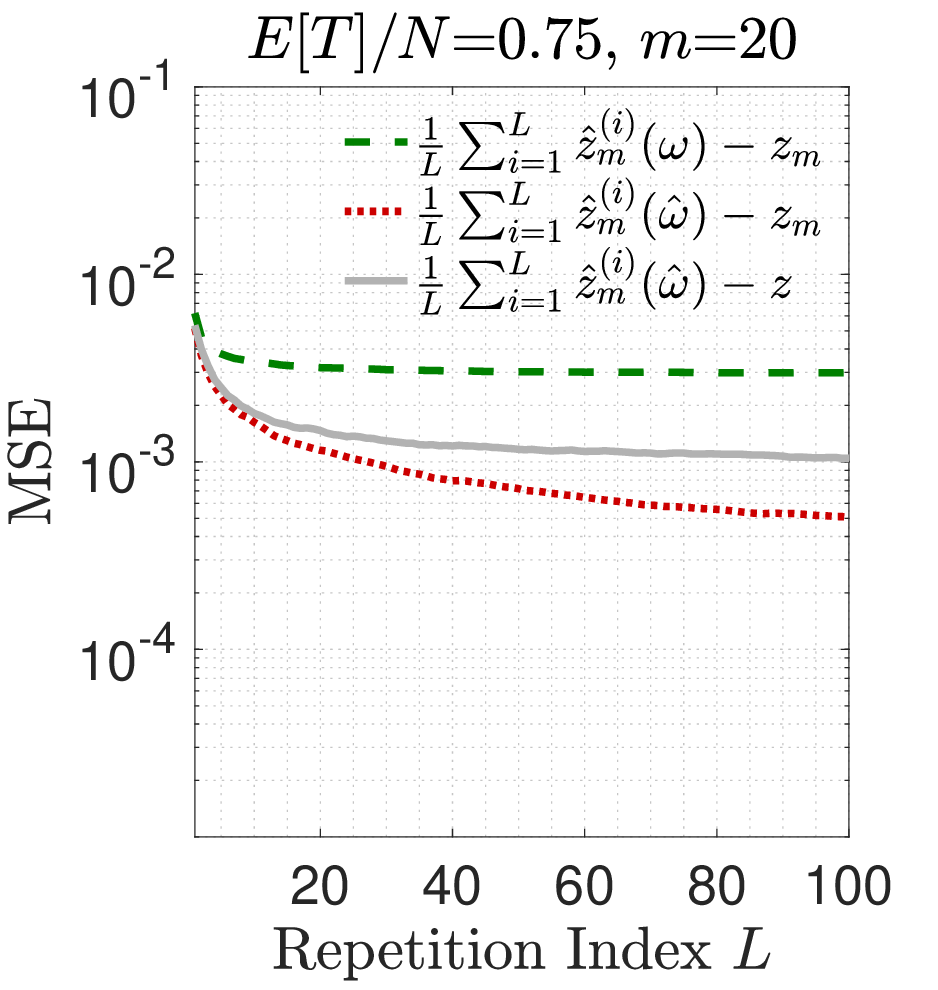}   
    \includegraphics[width=0.32\textwidth]{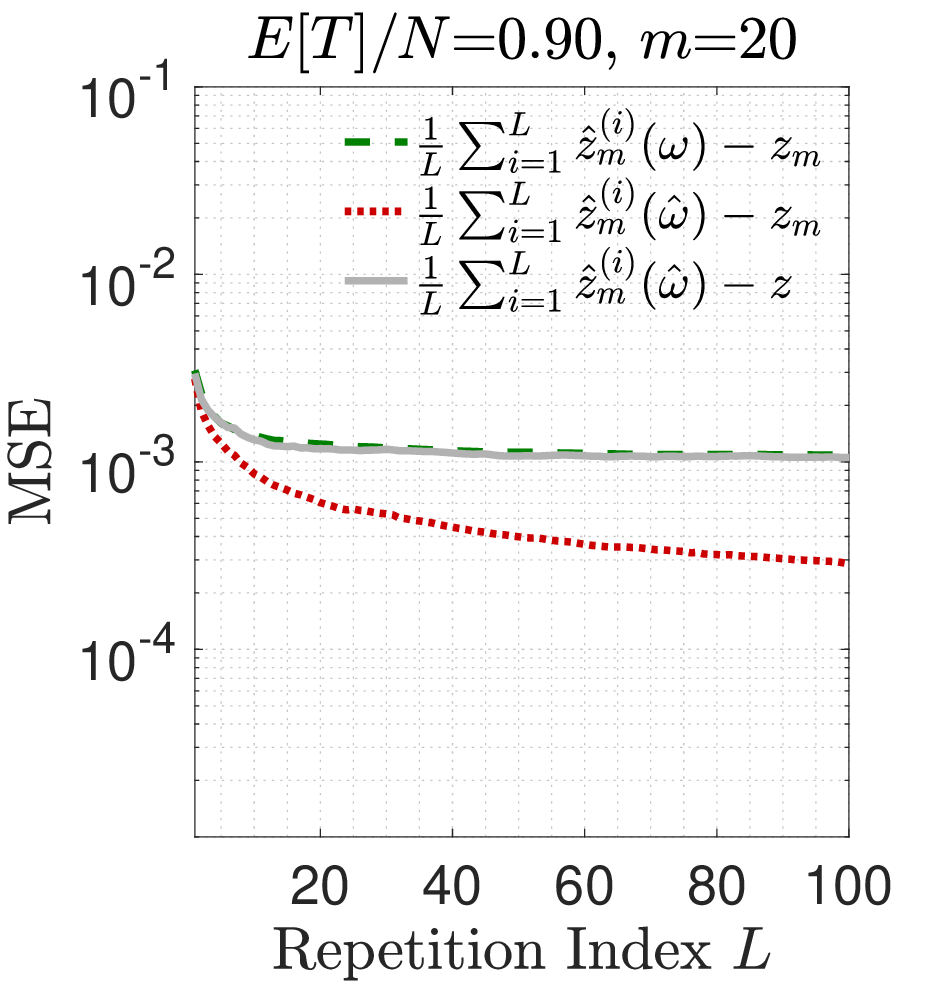}   
    \includegraphics[width=0.32\textwidth]{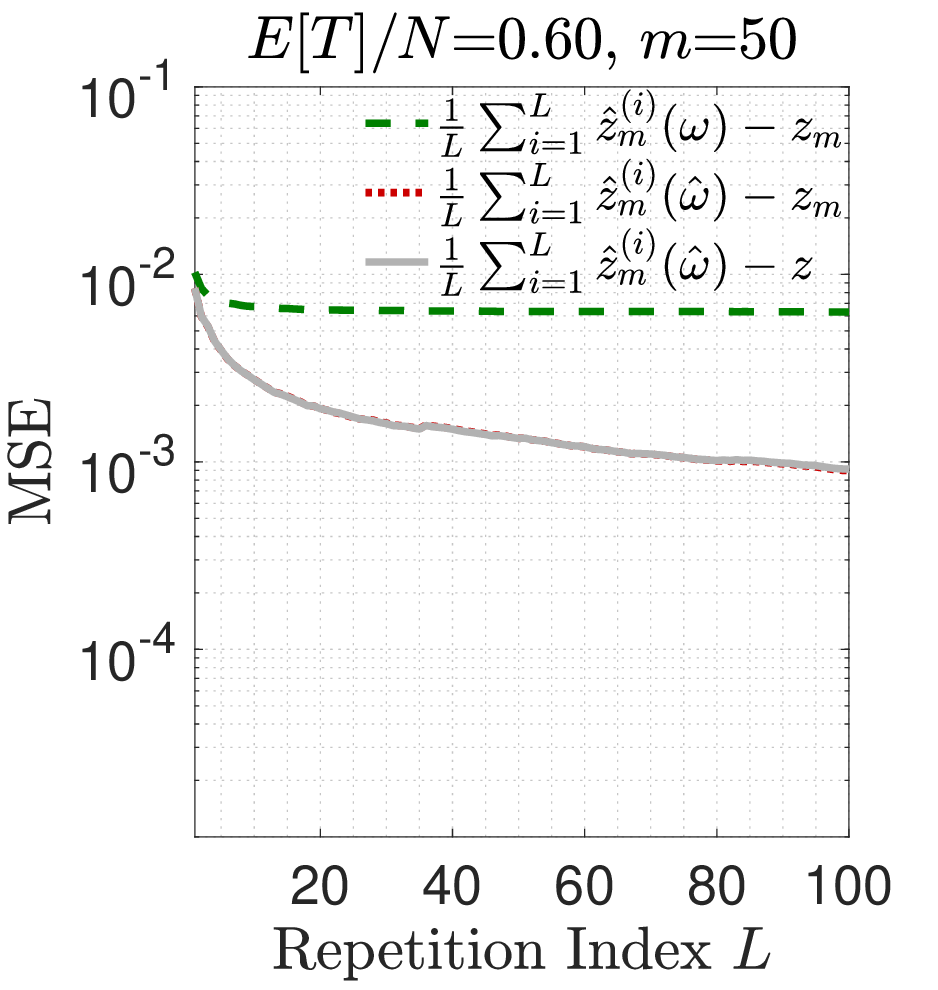}   
    \includegraphics[width=0.32\textwidth]{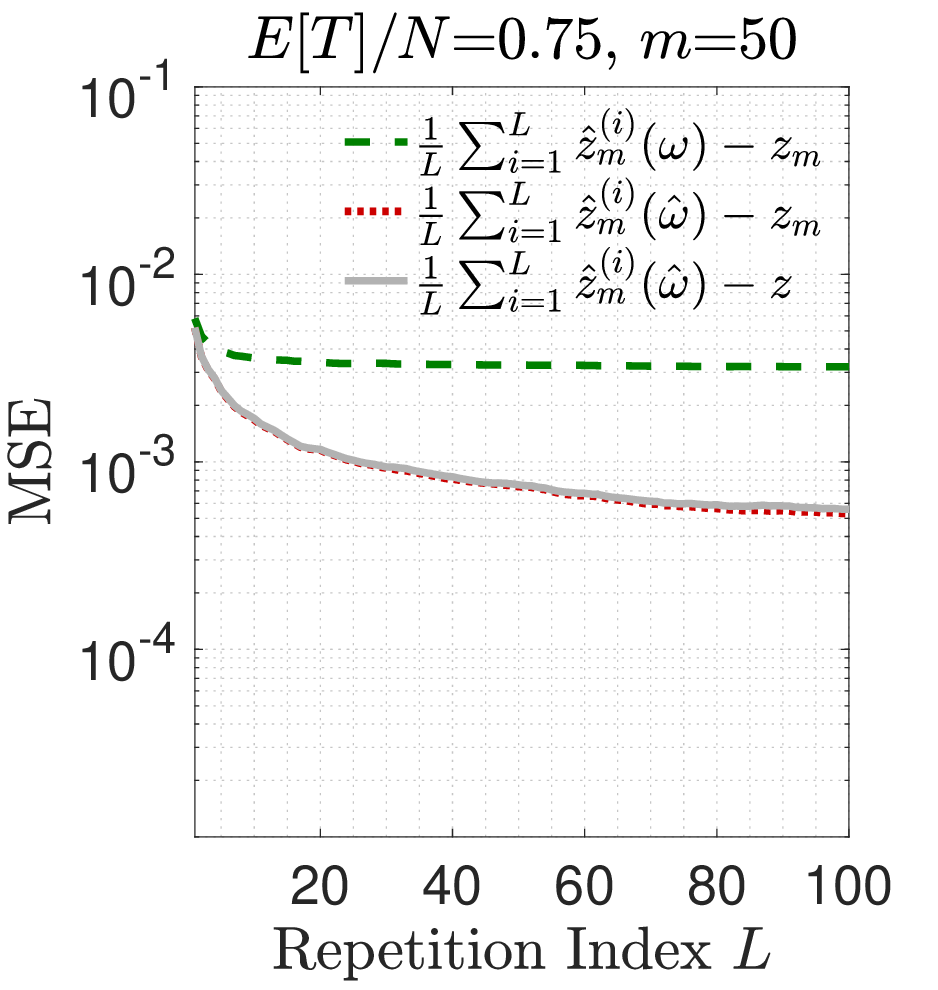}   
    \includegraphics[width=0.32\textwidth]{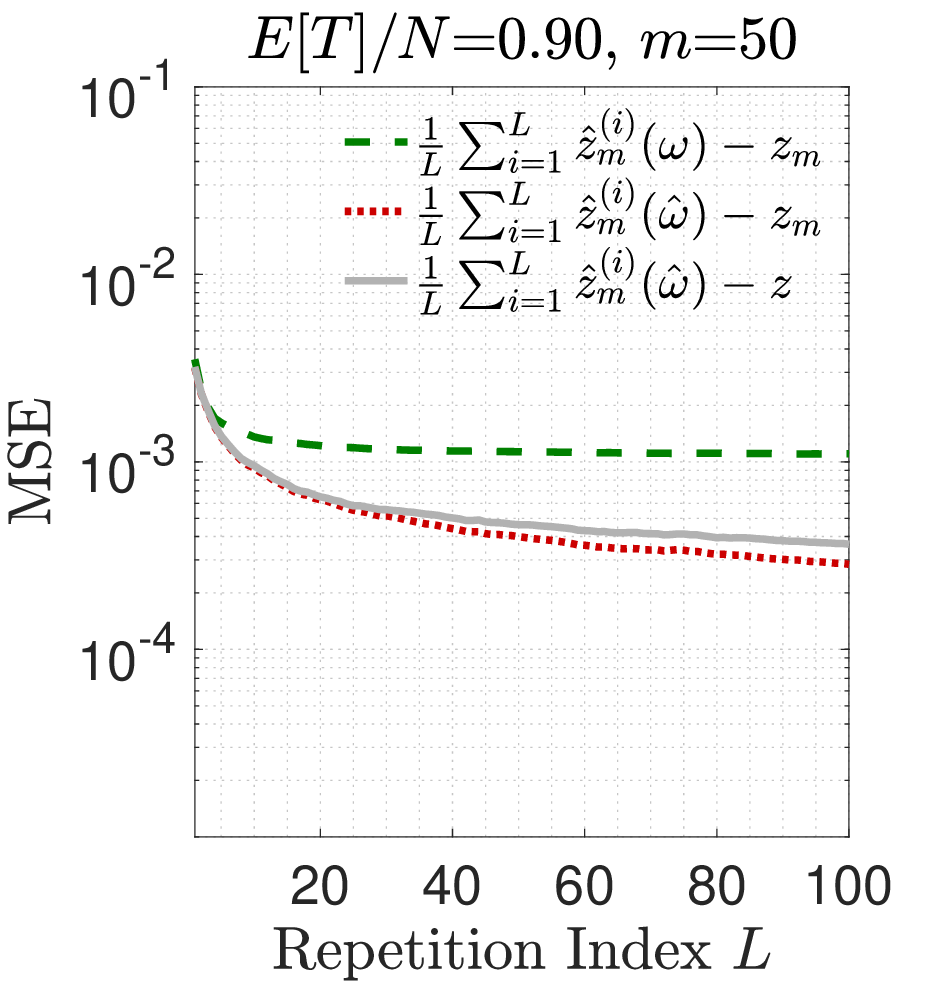}   
    \caption{MSE of $\frac{1}{L}\sum_{i=1}^L\widehat{z}_m^{(i)}-z_m$ on up to $L=100$ trials of Algorithm \ref{alg:richardson} for the matrix \textsc{bundle1} as well as  MSE of the quantity $\frac{1}{L}\sum_{i=1}^L\widehat{z}_m^{(i)}-z$ where $z$ is the solution of $Az=v$.}\label{fig2}
\end{figure}

\subsubsection{Performance on model problems}

We consider the numerical solution of Poisson's equation on the unit 
cube with Dirichlet boundary conditions:
\begin{equation} \label{eq:Dir}
\Delta f = h\ {\rm in}\ \Omega:= (0,1)\times (0,1)\times (0,1),\ \ 
u_{|\partial \Omega} = 0,
\end{equation}
where $\Delta$ denotes the Laplace operator and $\partial \Omega$ denotes the 
boundary of the unit cube. We discretize \eqref{eq:Dir} via 7-point stencil 
finite differences using the same mesh size along each dimension. The resulting 
coefficient matrix $A$ is SPD.

\begin{figure}
\centering
    \includegraphics[width=0.42\textwidth]{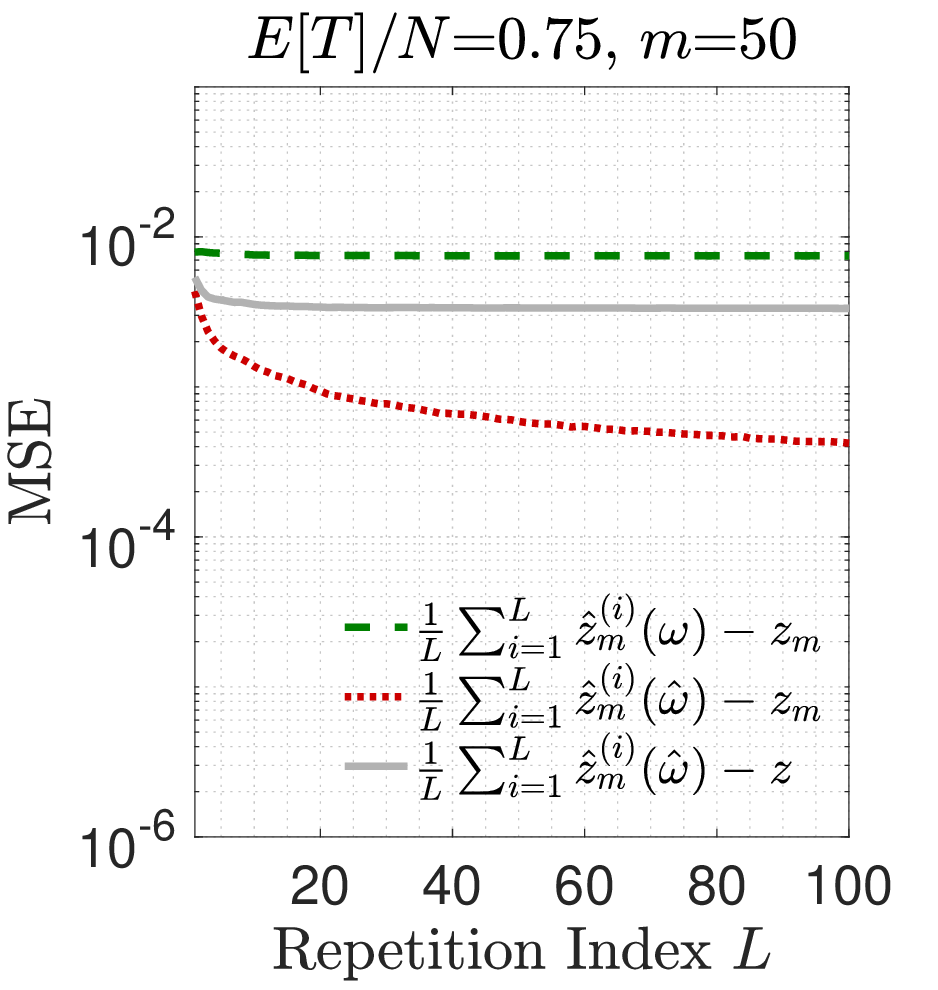}   
    \includegraphics[width=0.42\textwidth]{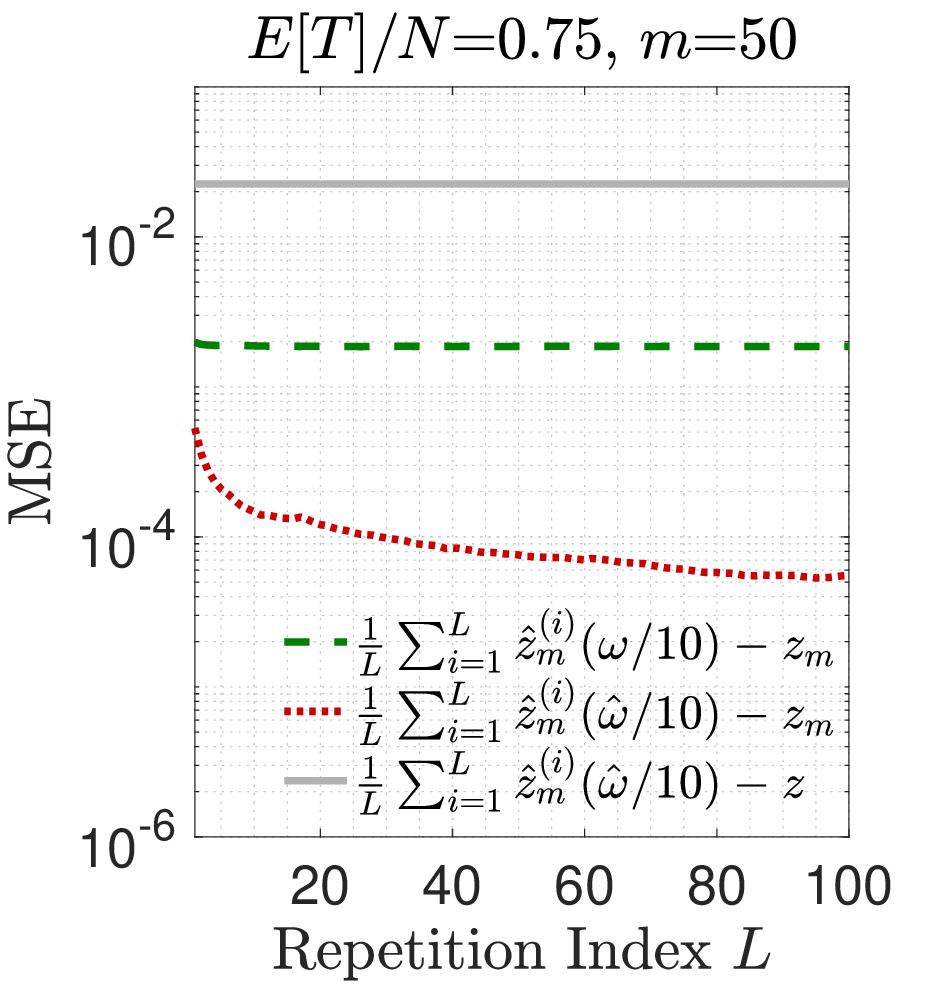}   
    \includegraphics[width=0.42\textwidth]{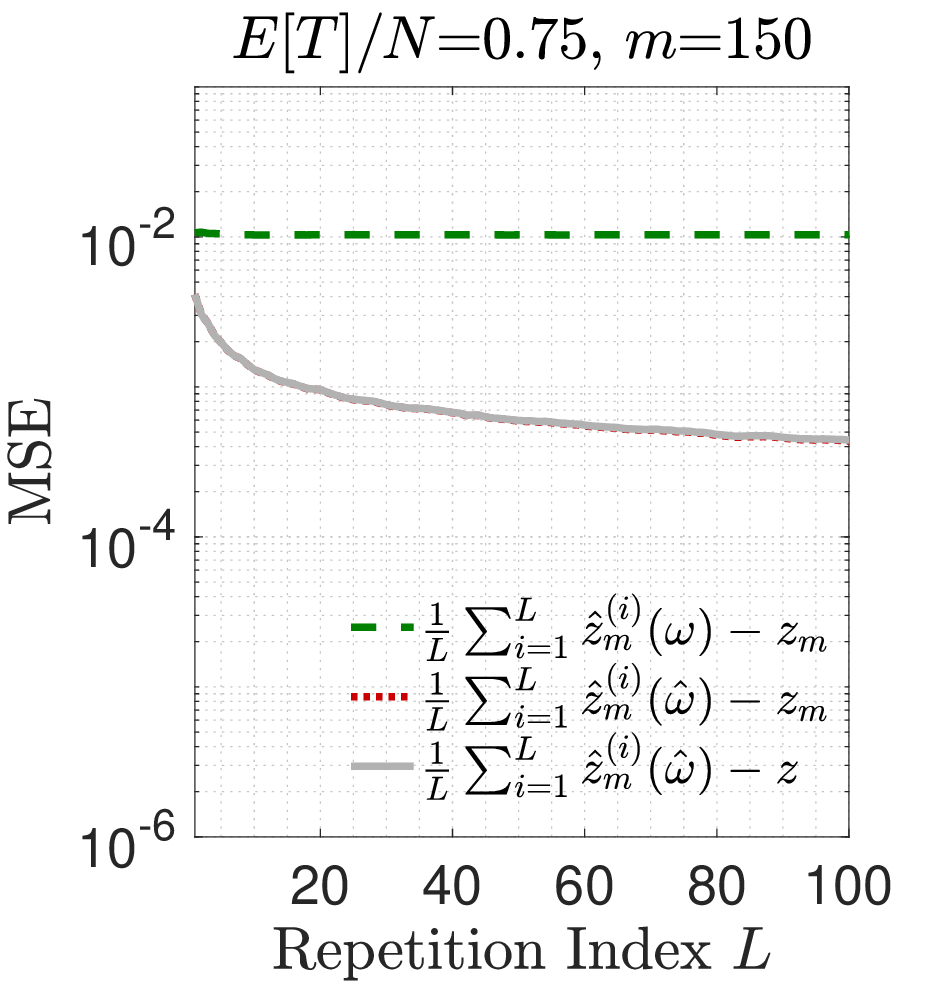}   
    \includegraphics[width=0.42\textwidth]{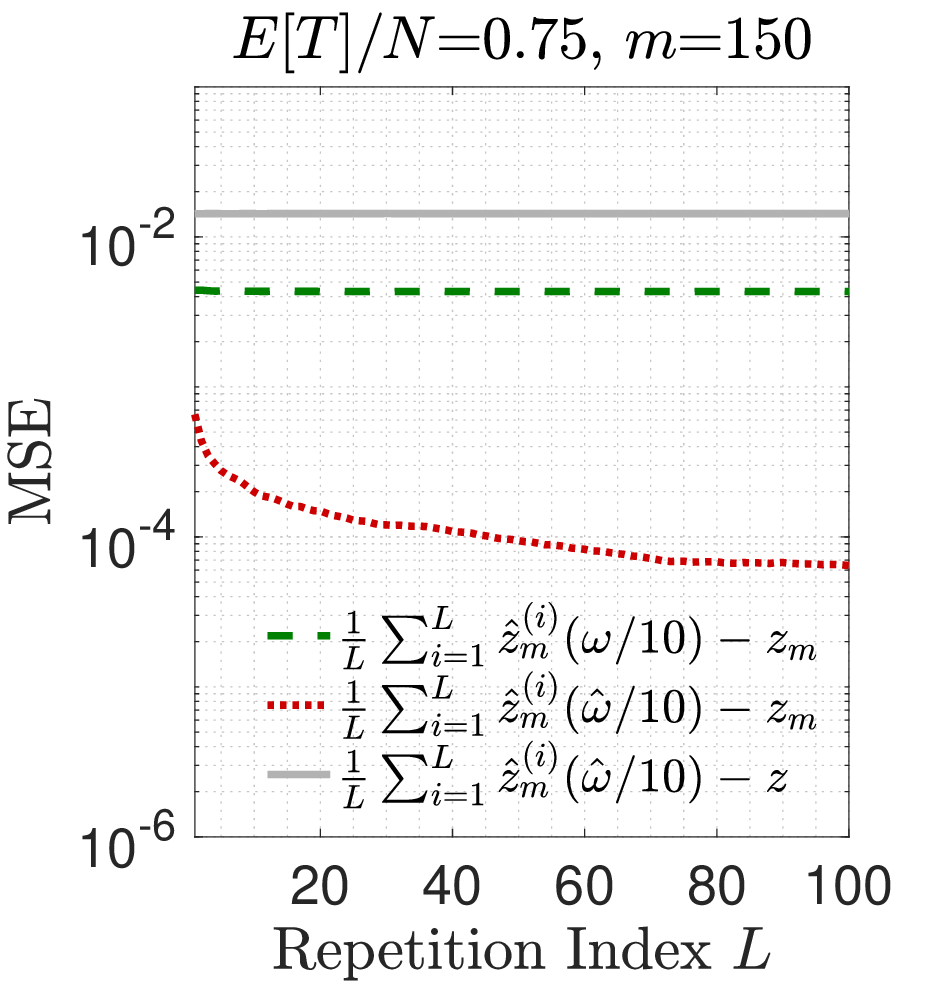}   
    \caption{Effects of reducing $\omega,\ \widehat{\omega}$ from their optimal values.}\label{fig3}
\end{figure}
The left two subfigures in Figure \ref{fig3} plot the MSE of the error vector $\frac{1}{L}\sum_{i=1}^L\widehat{z}_m^{(i)}-z_m$ 
when $\omega$ (dashed line) and $\widehat{\omega}$ (dashed-dotted line) as 
well as the MSE of the error vector $\frac{1}{L}\sum_{i=1}^L\widehat{z}_m^{(i)}-z$ (solid line) for 
the case $\tau=0.75$ and $m\in \{50,150\}$. In addition, we also 
test the dependence of the variance of $\widehat{z}_m$ on $\omega$. We repeat the experiment using the same settings 
except that now we decrease both $\omega$ and $\widehat{\omega}$ by an order of magnitude. 
The right two subfigures in Figure \ref{fig3} show that  the MSE of $\frac{1}{L}\sum_{i=1}^L\widehat{z}_m^{(i)}-z_m$ decreases for the same set of 
parameters when $\omega$ decreases. On the other hand, since classical Richardson 
iteration uses the non-optimal scalar parameter $\omega_{\rm CR}/10$ instead of 
$\omega_{\rm CR}$, the approximation of $z$ by $z_m$ is not as accurate as before. Therefore, 
although the MSE of $\frac{1}{L}\sum_{i=1}^L\widehat{z}_m^{(i)}-z_m$ decreases, the error between 
$\frac{1}{L}\sum_{i=1}^L\widehat{z}_m^{(i)}$ and $z$ actually becomes larger.

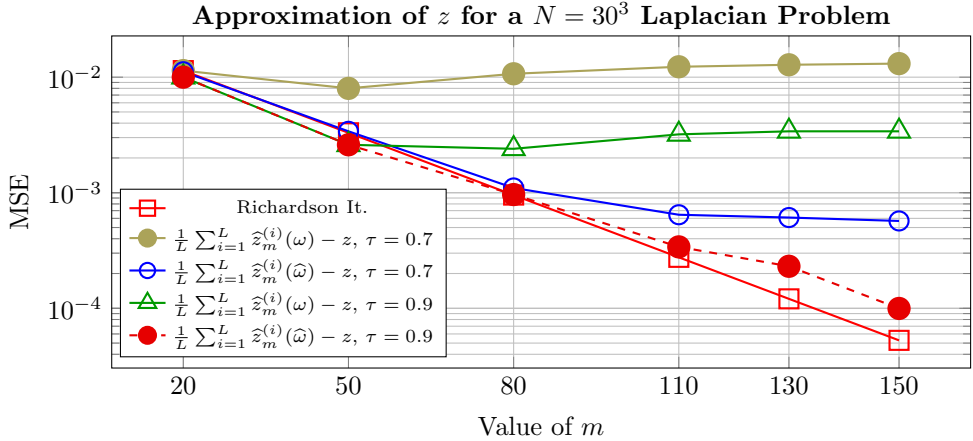
\begin{figure}
\centering
\begin{tikzpicture}
\begin{semilogyaxis}[
	xlabel={Value of $m$},
	ylabel={MSE},
        grid=both,minor tick num=1,
        every axis plot/.append style={thick},
        xtick={20,50,80,110,130,150},
        legend style={at={(0.39,0.54)},{font=\scriptsize}},
        title={{\bf Approximation of $z$ for a $N=30^3$ Laplacian Problem}},
        legend image post style={scale=0.8},
        title style={yshift=-1.5ex,},
        title style={yshift=-.1ex,},
        width=13cm, 
        height=6cm,
]
\addplot[red,mark=square,mark options=solid,mark size=3.6pt] coordinates {
(20,0.0114) (50,0.0033) (80,9.5438e-04)
	(110,2.7604e-04) (130,1.2073e-04) (150,5.2801e-05)
};

\addplot[black!40!yellow,mark=*,mark options=solid,mark size=4pt] coordinates {
(20,0.0114) (50,0.0080) (80,0.0107)
	(110,0.0123) (130,0.0128) (150,0.0131)
};
\addplot[blue,mark=o,mark options=solid,mark size=3.6pt] coordinates {
(20,0.01114) (50,0.0034) (80,0.0011)
	(110,6.45e-04) (130,6.1e-04) (150,5.7e-04)
};

\addplot[black!40!green,mark=triangle,mark options=solid,mark size=4.8pt] coordinates {
(20,0.01) (50,0.0026) (80,0.0024)
	(110,0.0032) (130,0.0034) (150,0.0034)
};
\addplot[black!10!red,mark=*,mark options=solid,mark size=4pt,dashed] coordinates {
(20,0.01) (50,0.0026) (80,9.6438e-04)
	(110,3.4e-04) (130,2.31e-04) (150,9.97e-05)
};
\legend{Richardson It., {$\frac{1}{L}\sum_{i=1}^L\widehat{z}_m^{(i)}(\omega)-z$, $\tau=0.7$}, 
 {$\frac{1}{L}\sum_{i=1}^L\widehat{z}_m^{(i)}(\widehat{\omega})-z$, $\tau=0.7$}, 
  {$\frac{1}{L}\sum_{i=1}^L\widehat{z}_m^{(i)}(\omega)-z$, $\tau=0.9$}, 
   {$\frac{1}{L}\sum_{i=1}^L\widehat{z}_m^{(i)}(\widehat{\omega})-z$, $\tau=0.9$}}
\end{semilogyaxis}
\end{tikzpicture}\vspace{-.08in}
\caption{MSE of the difference between the exact solution $z$ and the approximate solutions returned by classical Richardson iteration 
(with a scalar parameter $\omega$) 
and the difference between the exact solution $z$ and the sample mean of Algorithm \ref{alg:richardson} for a model problem. The number of samples generated 
by Algorithm \ref{alg:richardson} is fixed to $L=10$.}\label{fig0000}
\end{figure}
Figure \ref{fig0000} plots the MSE of $\frac{1}{L}\sum_{i=1}^L\widehat{z}_m^{(i)}-z$ and the error 
vector $z-z_m$ as a function of $m$. Here, we fix the number of the samples $L=10$. Similar to the previous results, we can 
observe that  Algorithm \ref{alg:richardson} with the scalar parameter $\widehat{\omega}$ 
converges -in expectation- to the same approximation that classical Richardson 
generates after $m$ steps. These results are in contrast to those obtained by setting the scalar parameter 
in Algorithm \ref{alg:richardson} equal to $\omega$ which stagnates. 
As is also indicated by our analysis, Algorithm \ref{alg:richardson} can approximate 
$z_m$ more accurately (for a fixed number of trials) for small values of $m$ due to the reduced variance.

\subsection{{Straggler-tolerant} second-order iterations}

We conclude this section with an illustration of the performance 
of classical and {straggler-tolerant} Chebyshev 
semi-iterative method on two sparse problems. 
For classical Chebyshev semi-iterative method, we consider the $2\times 2$ augmented 
system in (\ref{cheb3}) while for the {straggler-tolerant} version we consider the $2\times 2$ augmented system in (\ref{cheb2}). The scalars $\eta$ and 
$\nu$ are set as suggested in Section \ref{sec:cheby} with $\alpha=0.9\lambda_1$ and 
$\beta=1.1\lambda_N$, while we used the same random initial guess for each system.

Figure \ref{fig22} plots the MSE of the difference between the approximation returned after $m$ steps of 
classical Chebyshev semi-iterative method and the approximation returned after 
{straggler-tolerant} Chebyshev semi-iterative method averaged over three separate trials, 
for the solution of a linear system with the sparse matrix problem \textsc{crystm01}. 
For {straggler-tolerant} Chebyshev semi-iterative method, we consider two separate sampling 
rates, $\tau=0.7$ and $\tau=0.9$. In agreement with the results reported for the case of 
{straggler-tolerant} Richardson, the sample mean converges to the iterates generated by classical Chebyshev semi-iterative method and  higher values of $\tau$ lead to a greater error reduction for the same value of $m$.

Figure \ref{fig11} plots the convergence of Chebyshev semi-iterative method on the same 
$N=30^3$ Laplacian model problem shown in Figure \ref{fig0000}. By comparing these two figures, we can see that Chebyshev semi-iterative method converges faster than Richardson iteration for both the classical and {straggler-tolerant} variants.

\begin{figure}
\centering
\begin{tikzpicture}
\begin{semilogyaxis}[
	xlabel={Value of $m$},
	ylabel={MSE},
        grid=both,minor tick num=1,
        every axis plot/.append style={thick},
        xtick={5,10,15,20,25,30,35,40,45,50},
        legend style={at={(0.37,0.41)},{font=\scriptsize}},
        title={{\bf Approximation of $d$ for the Problem \textsc{crystm01}}},
        legend image post style={scale=0.8},
        title style={yshift=-1.5ex,},
        title style={yshift=-.1ex,},
        width=13cm, 
        height=6cm,
]
\addplot[red,mark=square,mark options=solid,mark size=3.6pt] coordinates {
(5,0.0146) (10,0.008) (15,0.0021)	(20,0.0017) (25,0.0012) (30,4.2226e-04) 
(35,1.8903e-04) (40,1.6772e-04) (45,7.4989e-05) (50,2.6e-05)
};

\addplot[blue,mark=o,mark options=solid,mark size=3.6pt] coordinates {
(5,0.06) (10,0.04) (15,0.021) (20,0.0087) (25,0.0052) (30,0.0038) 
(35,0.002) (40,9.9e-04) (45,7.0e-04) (50,5.92e-04)
};

\addplot[black!10!red,mark=*,mark options=solid,mark size=4pt,dashed] coordinates {
(5,0.03) (10,0.01) (15,0.0085) (20,0.0055) (25,0.003) (30,0.0012) 
(35,7.8e-04) (40,4.9e-04) (45,2.6e-04) (50,1.1e-04)
};
\legend{Chebyshev 2nd order, 
 {$\frac{1}{L}\sum_{i=1}^L\widehat{d}_m^{(i)}-d$, $\tau=0.7$}, 
   {$\frac{1}{L}\sum_{i=1}^L\widehat{d}_m^{(i)}-d$, $\tau=0.9$}}
\end{semilogyaxis}
\end{tikzpicture}\vspace{-0.08in}
\caption{MSE of the difference between the exact solution $d$ and the approximation 
$d_m$ returned by classical Chebyshev semi-iterative method, and the difference between the exact solution $d$ and the sample mean 
of the $m$-step approximations returned by {straggler-tolerant} Chebyshev semi-iterative method, for the matrix problem \textsc{crystm01}.  The number of samples is fixed to $L=3$.}\label{fig22}
\end{figure}
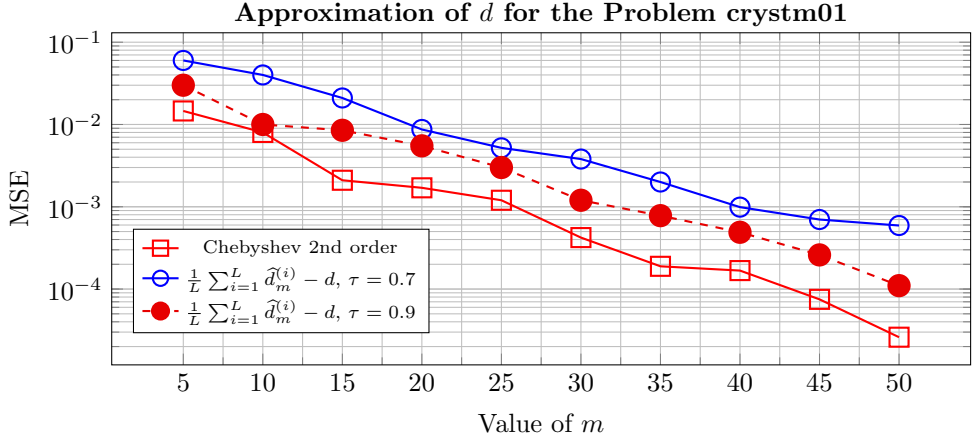

\begin{figure}
\centering
\begin{tikzpicture}
\begin{semilogyaxis}[
	xlabel={Value of $m$},
	ylabel={MSE},
        grid=both,minor tick num=1,
        every axis plot/.append style={thick},
        xtick={5,10,15,20,25,30,35,40,45,50},
        legend style={at={(0.39,0.5)},{font=\scriptsize}},
        title={{\bf Approximation of $d$ for a $N=30^3$ Laplacian Problem}},
        legend image post style={scale=0.8},
        title style={yshift=-1.5ex,},
        title style={yshift=-.1ex,},
        width=13cm, 
        height=6cm,
]
\addplot[red,mark=square,mark options=solid,mark size=3.6pt] coordinates {
(5,0.0043) (10,0.003) (15,0.0016)	(20,8.6228e-04) (25,4.3789e-04) (30,2.0766e-04) 
(35,8.9757e-05) (40,3.7211e-05) (45,1.7185e-05) (50,9.5e-06)
};

\addplot[blue,mark=o,mark options=solid,mark size=3.6pt] coordinates {
(5,0.011) (10,0.007) (15,0.0094) (20,0.0058) (25,0.0046) (30,0.0025) 
(35,0.001) (40,8.7e-04) (45,4.8e-04) (50,2.1e-04)
};

\addplot[black!10!red,mark=*,mark options=solid,mark size=4pt,dashed] coordinates {
(5,0.0073) (10,0.005) (15,0.0036) (20,0.0028) (25,9.3e-04) (30,8.1e-04) 
(35,5.4e-04) (40,2.1e-04) (45,9.2e-05) (50,4.9e-05)
};
\legend{Chebyshev 2nd order, 
 {$\frac{1}{L}\sum_{i=1}^L\widehat{d}_m^{(i)}-d$, $\tau=0.7$}, 
   {$\frac{1}{L}\sum_{i=1}^L\widehat{d}_m^{(i)}-d$, $\tau=0.9$}}
\end{semilogyaxis}
\end{tikzpicture}\vspace{-0.08in}
\caption{MSE of the difference between the exact solution $d$ and the approximation 
$d_m$ returned by classical Chebyshev semi-iterative method, and the difference between the exact solution $d$ and the sample mean 
of the $m$-step approximations returned by {straggler-tolerant} Chebyshev semi-iterative method, for a Laplacian problem.  The number of samples is fixed to $L=3$.}\label{fig11}
\end{figure}
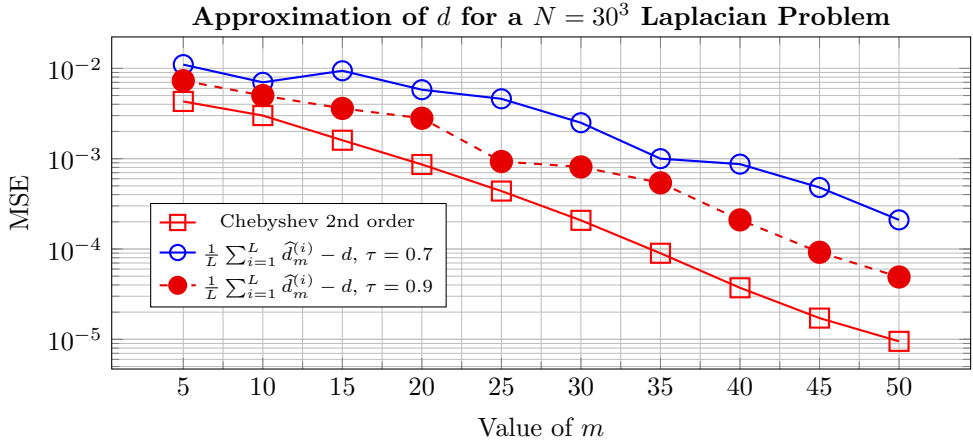

\section{Conclusion} \label{sec6}

In this paper, we considered the solution of linear systems under the 
constraint that matrix-vector products with the iteration matrix $A$ 
are performed through an oracle which returns $T\in [1,N]$ entries 
of the matrix-vector product while replacing the $N-T$ non-returned ones by 
zero. The $T$ returned entries are indexed by the row subset ${\cal T}$ 
where $|{\cal T}|=T$. By assumption, both $T$ and ${\cal T}$ are random 
variables. Our theoretical results indicate that {straggler-tolerant} Richardson iteration and Chebyshev semi-iterative method equipped 
with partial matrix-vector products can still converge -in expectation- to 
the solution obtained by classical ones if one scalar parameter 
is weighted by the scalar $N/\mathbb{E}[T]$ when ${\cal T}$ follows the uniform distribution. Numerical experiments on model and sparse problems verified the theoretical aspects of the proposed algorithms.

In our future work, we aim to investigate the convergence of {straggler-tolerant} solvers in non-uniform distribution cases and provide a more rigorous analysis of the variance in the iterates. Moreover, we plan to study the application of {straggler-tolerant} solvers as  
inner preconditioners in flexible Krylov methods. In this case, both the preconditioner and the matrix-vector product with the matrix $A$ can be computed less accurately as the approximate solution becomes more accurate, which in turn, implies that there might be room to sample from a distribution with lower $\mathbb{E}[T]$ dynamically, i.e., the framework can tolerate more straggling workers. We also plan to extend our development of straggler-tolerant linear solvers to nonlinear ones \cite{nltgcr,anderson24}. 

\section*{Acknowledgement} 
{The authors would like to thank two anonymous referees for their valuable suggestions
which greatly improve the presentation and the scope of the paper.}

\bibliographystyle{siamplain}
\bibliography{references}

\end{document}